\documentclass[11pt]{article}
\usepackage{amsmath,amsthm, amssymb, amsfonts,accents,nccmath}
\usepackage{enumitem}
\usepackage{array}
\usepackage[toc,page]{appendix}
\usepackage[english]{babel}
\usepackage{dsfont}
\usepackage{booktabs}
\usepackage[linesnumbered,commentsnumbered,ruled,vlined]{algorithm2e}
\usepackage{tikz}
\usepackage{tikz-network}
\usepackage[utf8]{inputenc}
\usepackage[english]{babel}

\usepackage{caption}
\usepackage{subcaption}

\usepackage{hyperref}
\hypersetup{
    colorlinks=true,
    linkcolor=blue,
    filecolor=darkmagenta,      
    urlcolor=cyan,
    citecolor=red
} 
\makeatletter
\def\BState{\State\hskip-\ALG@thistlm}
\makeatother
\numberwithin{equation}{section}

\newtheorem{theorem}{Theorem}[section]

\newtheorem{lem}[theorem]{Lemma}
\newtheorem{prop}[theorem]{Proposition}
\newtheorem{defn}[theorem]{Definition}
\newtheorem{rem}[theorem]{Remark}
\newtheorem{ex}[theorem]{Example}

\newcommand{\bigzero}{\mbox{\normalfont\Large\bfseries 0}}

\newcommand{\ds}{\displaystyle}

\newcommand{\cof}{\textup{cof }}

\title{Distance Matrix of a  Multi-block Graph: Determinant and Inverse }
\author{Joyentanuj Das$^*$   \quad and \quad Sumit Mohanty\footnote{School of Mathematics, IISER Thiruvananthapuram, Maruthamala P.O., Vithura, 
Thiruvananthapuram,\newline \indent   Kerala- 695 551, India.
 \newline \indent Emails:
joyentanuj16@iisertvm.ac.in,  \quad sumit@iisertvm.ac.in, sumitmath@gmail.com. 
\newline \indent Phone: +91-471-2778112.    }}

\date{}

\begin{document}

\maketitle

\begin{abstract}
A connected graph is called a multi-block graph if each of its blocks is a complete multipartite graph. Building on the work of~\cite{Bp3,Hou3}, we compute the determinant and inverse of the distance matrix for a class of multi-block graphs.

\end{abstract}

\noindent {\sc\textsl{Keywords}:} $m$-partite graphs, Laplacian-like matrix, Distance matrix,  Determinant, Cofactor.

\noindent {\sc\textbf{MSC}:}  05C12, 05C50
\section{Introduction  and Motivation}
Let $G=(V(G),E(G))$  be a finite, simple, connected graph with $V(G)$ as the  set of vertices and $E(G)\subset V(G) \times V(G)$ as the set of edges in $G$, we simply write $G=(V,E)$ if there is no scope of confusion. We write $i\sim j$ to indicate that the vertices $i,j \in V$ are adjacent in $G$.  For $m\geq 2$,  a graph is said to be  $m$-partite   if the vertex set can be partitioned into $m$  subsets $V_i, \ 1\leq i\leq m$  with $|V_{i}|=n_i$ and $|V|=\sum_{i=1}^m n_i$ such that $ E\subset \bigcup_{i,j \atop i\neq j} V_{i} \times V_{j}$.  A $m$-partite graph is said to be a complete $m$-partite graph, denoted by $K_{n_1,n_2,\cdots,n_m}$ if every vertex in $V_{i}$ is adjacent to every vertex of $V_{j}$ and vice  versa for $i \neq j$ and $i,j = 1,2,\ldots,m$.  A graph with $m$ vertices is called complete, if each vertex of the graph is adjacent to every other vertex and is denoted by $K_m$.  The complete graph $K_m$ ($m\geq 2$) can also be seen as  complete $m$-partite graph  $K_{n_1,n_2,\cdots,n_m}$ with $n_i=1$ for all $1\leq i\leq m$.
 
Before proceeding further, we first introduce a few notations which will be used time and again throughout this article. Let $I_n$, $ \mathds{1}_n$ and $e_i$ denote the identity matrix, the column vector of all ones  and the column vector with $1$ at the $i^{th}$ entry,  respectively.  Further, $J_{m \times n}$  denotes the $m\times n$ matrix of all ones and if $m=n$, we use the notation $J_m$.  We write $\mathbf{0}_{m \times n}$ to represent zero matrix of order $m \times n$ and simply write $\mathbf{0}$ if there is no scope of confusion with the order of the matrix.  We use $\mathbb{N}$ to represent the set of natural numbers.

Let $A$ be an $m \times n$ matrix. We use the notation $A(i \mid j)$ to denote the submatrix obtained by deleting the $i^{th}$ row and the $j^{th}$ column. Given a matrix $A$, we use $A^t$ to denote the transpose of the matrix $A$. Let $B$ be an $n \times n$ matrix. For $1 \leq i,j \leq n$, the cofactor $c_{ij}$ is defined as $(-1)^{i+j} \det B(i \mid j)$. The transpose of cofactor matrix $[c_{ij}]$  of  $B$ is called the adjoint of $B$, denoted by $\mbox{Adj }B$. We use the notation  $\cof B$ to denote the sum of all cofactors of $B$. Hence $\cof B=\mathds{1}_n^t \mbox{ Adj }B \ \mathds{1}_n$.

A connected graph $G$  is a metric space with respect to the metric $d,$ where $d(i,j)$ equals the length of the shortest path between vertices $i$ and $j$. Note that $d(i,i)=0$. Before proceeding further, we recall the definitions of the \emph{distance matrix} and the \emph{Laplacian matrix} of a graph $G$.

Let $G$ be a graph with $n$ vertices. The distance matrix of   graph  $G$  is an  $n \times n$ matrix,  denoted by  $D(G) = [d_{ij}]$, where $d_{ij}= d(i,j)$,  and the  Laplacian matrix of $G$  is an $n \times n$ matrix, denoted as $L(G)=[l_{ij}]$, where
$$
l_{ij}=
\begin{cases}
\delta_i & \text{if} \ i = j; \\
-1 &  \text{if} \ i \neq j, i \sim j; \\
0 &\text{otherwise,}
\end{cases}
$$
and $\delta_i$ denotes the degree of the vertex $i$. It is well known that $L(G)$ is a symmetric, positive semi-definite matrix. The constant vector  $\mathds{1}$ is the eigenvector of $L(G)$ corresponding to the smallest eigenvalue $0$ and hence  satisfies $L(G)\mathds{1} = \mathbf{0}$ and $\mathds{1}^t L(G) = \mathbf{0}$ (for details see~\cite{Bapat}).

Let $T$ be a tree with $n$ vertices. In~\cite{Gr1}, the authors proved that the determinant of the distance matrix $D(T)$ of $T$ is given by $\det D(T)=(-1)^{n-1}(n-1)2^{n-2}.$ Note that, the determinant does not depend on the structure of the tree, but it only depends on the number of vertices. In \cite{Gr2}, it was shown that the inverse of the distance matrix of tree $T$ is given by $D(T)^{-1} = -\dfrac{1}{2}L(T) + \dfrac{1}{2(n-1)}\tau \tau^t,$ where $\tau = (2-\delta_1,2-\delta_2,...,2-\delta_n)^t.$  The above expression gives a formula for inverse of the distance matrix of a tree in terms of the Laplacian matrix.

 A vertex $v$ of a graph $G$ is a cut-vertex of $G$ if $G - v$ is disconnected. A block of the graph $G$ is a maximal connected subgraph of G that has no cut-vertex.   In~\cite{Bp3}, the authors compute the determinant and the inverse of the distance matrix of graphs whenever each of its blocks is a complete graph, called such graphs as block graphs. Further, in~\cite{Hou3}  the determinant and the inverse of the distance matrix of graphs were computed whenever each of its blocks is a complete bipartite graph,  such graphs are called  bi-block graphs.  To be specific, the authors defined a matrix  $\mathcal{L}$ satisfying  $\mathcal{L}\mathds{1} = \mathbf{0}$ and $\mathds{1}^t \mathcal{L} = \mathbf{0}$ and such a matrix was formally defined as a Laplacian-like matrix in~\cite{Zhou1}.  Further, it was shown that for a given graph $G$ of the above classes if the determinant of the distance matrix $ D(G)$ is not zero, then the inverse of $D(G)$  is given by
$$D(G)^{-1} = -\mathcal{L} + \frac{1}{\lambda_{G}}\mu \mu^t,$$
where $\mu$ is a column vector, $\lambda_G$ is a suitable constant and $\mathcal{L}$ is a Laplacian-like matrix (for details see~\cite{Bp3,Hou3}). Similar results were obtained for other class of graphs, namely catcoid digraphs, cycle-clique graphs, weighted cactoid digraph (for details see~\cite{Hou1, Hou2,Zhou2}).

In this manuscript, our  objective is to extend the results in~\cite{Bp3,Hou3}. To be precise, we aim to compute the determinant and inverse of the distance matrix of graphs such that each of its blocks is a complete multipartite graph; we call such graphs as  multi-block graphs.  Now we will  recall a few preliminary results useful for subsequent sections.

\begin{lem}\label{Lem:cof}\cite{Bapat}
Let $A$ be an $n \times n$ matrix. Let $M$ be the matrix obtained from $A$ by subtracting the first row from all other rows and then subtracting the first column from all other columns. Then $$\cof A = \det M(1 | 1).$$
\end{lem}

A graph $G$ is said to be $k$-connected if there does not exist a set with less than $k$ vertices whose removal disconnects the graph $G$. The authors in~\cite{Gr3} established the following result for connected graphs with $2$-connected blocks.

\begin{theorem}\label{Thm:cof-det}\cite{Gr3}
Let $G$ be a connected graph with $2$-connected blocks $G_1,G_2,\cdots,G_b$. Then
$$\cof D(G) = \displaystyle \prod_{i=1}^{b} \cof D(G_i),$$
$$\det D(G) = \displaystyle \sum_{i=1}^{b} \det D(G_i) \prod_{j \neq i} \cof D(G_j).$$
\end{theorem}

Let $B$ be an $n\times n$ matrix partitioned as
\begin{equation}\label{eqn:B}
B= \left[
\begin{array}{c|c}
B_{11}& B_{12} \\
\midrule
B_{21} &B_{22}
\end{array}
\right],
\end{equation}
where $ B_{11}$ and $B_{22}$ are square matrices. If  $B_{22}$ is nonsingular, then the Schur complement of $B_{22}$ in $B$ is defined to be the matrix $B/B_{22}=B_{11} -B_{12}B_{22}^{-1}B_{21}$.

The next result gives us the inverse of a partitioned matrix using Schur complement, whenever the matrix is invertible.
\begin{prop}\label{prop:schur}\cite{Zhang1}
Let $B$ be a nonsingular matrix partitioned as in Eqn~(\ref{eqn:B}). If  $B_{22}$ is square and invertible, then
\[
B^{-1}=\left[
\begin{array}{c|c}
(B/B_{22})^{-1}& -(B/B_{22})^{-1}B_{12}B_{22}^{-1}  \\
\midrule
-B_{22}^{-1}B_{21}(B/B_{22})^{-1} & B_{22}^{-1}+B_{22}^{-1}B_{21}(B/B_{22})^{-1}B_{12}B_{22}^{-1}
\end{array}
\right].
\]
\end{prop}

Next we state a result without proof that gives the  eigenvalues and inverse for matrix of the  form $aI+bJ$.
\begin{lem}\label{lem:aI+bJ}
Let  $n \geq 2$   and $J_n$, $I_n$ be matrices as defined before. For $a \neq 0$, the eigenvalues of $aI_n+bJ_n$ are $a$ with multiplicity $n-1$, $a+nb$ with multiplicity $1$ and the determinant is given by $a^{n-1}(a+nb)$. Moreover, the matrix is invertible if and only if $a+nb \neq 0$ and the inverse is given by $$(aI_n + bJ_n)^{-1} = \frac{1}{a} \left(I_n - \frac{b}{a+nb} J_n\right).$$
\end{lem}
We conclude this section with a standard result on computing the determinant of block matrices.

\begin{prop}\label{prop:blockdet}~\cite{Zhang}
 Let $A_{11}
   \mbox{ and } A_{22}$ be square matrices.  Then 
   \begin{small}
   $$\det \left[
\begin{array}{c|c}
A_{11} & \bigzero  \\
\midrule
A_{21} & A_{22}
\end{array}
\right] = \det A_{11} \det A_{22}.$$
   \end{small}   
\end{prop}

This article is organized as follows.  In Sections~\ref{sec:Single-m-partitite} and~\ref{sec:inverse-Single-m-partitite}, we find the cofactor and determinant  of the distance matrix for a complete $m$-partite graph for $m\geq 3$ and  compute its inverse whenever it exists. In Section~\ref{sec:multi-block}, we compute the inverse of multi-block graphs subject to the condition that cofactor is nonzero.  Finally, Section~\ref{sec:T6-Tn-b} deals with a class of multi-block graphs, not covered in Section~\ref{sec:multi-block}.
\vspace*{-.3cm}

\section{Cofactor and Determinant  of $D(K_{n_1,n_2,\cdots,n_m})$}\label{sec:Single-m-partitite}
Let $D(K_{n_1,n_2,\cdots,n_m})$ be the distance matrix of complete $m$-partite graph $K_{n_1,n_2,\cdots,n_m}$. Then the matrix $D(K_{n_1,n_2,\cdots,n_m})$ can be expressed in the following block form 

\begin{small}
\begin{equation}\label{eqn:D(K)}
D(K_{n_1,n_2,\cdots,n_m}) = \left[\begin{array}{c|c|c|c}
2(J_{n_1} - I_{n_1}) & J_{n_1 \times n_2} & \cdots & J_{n_1 \times n_m}\\
\hline
J_{n_2 \times n_1} & 2(J_{n_2} - I_{n_2}) & \cdots & J_{n_2 \times n_m}\\
\hline
\vdots & \cdots  & \ddots &\vdots\\
\hline
J_{n_m \times n_1} & J_{n_m \times n_2} &  \cdots & 2(J_{n_m} - I_{n_m})\\
\end{array} \right].
\end{equation}
\end{small}

We begin with a lemma which will be used to compute the cofactor of $ D(K_{n_1,n_2,\cdots,n_m})$.

\begin{lem}\label{Lem:Det_3}
Let $C_m$ be a square matrix of order $m$ as following
\begin{small}
\begin{equation*}\label{eqn:M_3}
C_m = \left[\begin{array}{ccccc}
n_1  & 2(n_1 - 1)& 2(n_1 - 1) & \cdots & 2(n_1 - 1)\\
n_2 & 2 & n_2 &\cdots & n_2\\
n_3 & n_3 & 2 &\cdots & n_3\\
\vdots & \vdots  &\vdots  & \ddots &\vdots\\
n_m  & n_m & n_m &  \cdots & 2\\
\end{array} \right].
\end{equation*} 
\end{small}
The determinant of the above matrix is given by 
$$\det C_m = (-1)^{m-1}\sum_{i=1}^m \left( n_i\prod_{j \neq i}(n_j - 2) \right).$$
\end{lem}

\begin{proof}
Subtracting the first column form all the remaining columns of $C_m$ yields the following matrix:
\begin{small}
$$
\left[\begin{array}{ccccc}
n_1  & n_1-2 & n_1-2 & \cdots & n_1-2\\
n_2 & 2-n_2 & 0 &\cdots & 0\\
n_3 & 0 & 2-n_3  &\cdots & 0\\
\vdots & \vdots  &\vdots  & \ddots &\vdots\\
n_m  & 0 & 0 &  \cdots & 2-n_m \\
\end{array} \right].
$$
\end{small}
 Now expanding along the first row, we get 
$$\det C_m = n_1 \prod_{j=2}^m (2-n_m)-(n_1-2)\sum_{i=2}^m n_i \prod_{j\neq i}(2-n_j)$$
 and the desired result follows.
\end{proof}

\begin{theorem}\label{thm:cof}
Let $D(K_{n_1,n_2,\cdots,n_m})$ be the distance matrix of complete $m$-partite graph $K_{n_1,n_2,\cdots,n_m}$ on $|V|=\sum_{i=1}^m n_i$ vertices. Then the cofactor of the distance matrix is given by
$$\cof D(K_{n_1,n_2,\cdots,n_m}) = (-2)^{|V|-m}\left[ \sum_{i=1}^m \left( n_i\prod_{j \neq i}(n_j - 2) \right)\right].$$
\end{theorem}
\begin{proof}
For complete $m$-partite graph $K_{n_1,n_2,\cdots,n_m}$ with $n_i=1$ for all $1\leq i\leq m$, we have $K_{n_1,n_2,\cdots,n_m}=K_m$. Then, the result is true as $\cof D(K_m)= (-1)^{m-1} m$. For other cases, without loss of generality, let $n_1>1$ and  $M$ be the matrix obtained from $D(K_{n_1,n_2,\cdots,n_m})$ by subtracting the first row from all other rows and then subtracting the first column from all other columns. Then the block form of the matrix $M(1|1)$ is given by
\begin{small}
$$
 \left[\begin{array}{c|c|c|c}
-2(J_{n_1-1} + I_{n_1-1}) & -2J_{(n_1-1) \times n_2} & \cdots & -2J_{(n_1-1) \times n_m}\\
\hline
-2J_{n_2 \times (n_1-1)} &  -2I_{n_2} & \cdots & -J_{n_2 \times n_m}\\
\hline
\vdots & \cdots  & \ddots &\vdots\\
\hline
-2J_{n_m \times (n_1-1)} & -J_{n_m \times n_2} &  \cdots &  -2I_{n_m}\\
\end{array} \right].
$$
\end{small}
 First for each partition of $M(1|1)$ we subtract the first column from all other columns, then add all the rows to the first row. Further, we shift the first column of all the $m$-partition to the first  $m$ columns and repeat the same operation for the rows. Then the resulting matrix is of the following block form:
\begin{small}
$$
\left[
\begin{array}{c|c}
-\widetilde{C}_m& \mathbf{0} \\
\hline
* & -2I_{|V|-(m+1)}
\end{array}
\right], 
$$ where
$$
\widetilde{C}_m = \left[\begin{array}{ccccc}
2n_1  & 2(n_1 - 1)& 2(n_1 - 1) & \cdots & 2(n_1 - 1)\\
2n_2 & 2 & n_2 &\cdots & n_2\\
2n_3 & n_3 & 2 &\cdots & n_3\\
\vdots & \vdots  &\vdots  & \ddots &\vdots\\
2n_m  & n_m & n_m &  \cdots & 2\\
\end{array} \right].$$
\end{small}
Using  Proposition~\ref{prop:blockdet}, Lemmas~\ref{Lem:cof} and \ref{Lem:Det_3} the result follows.
\end{proof}

The determinant of $D(K_{n_1,n_2,\cdots,n_m})$ can be obtained using elementary matrix operations similar to the computation of cofactor of $D(K_{n_1,n_2,\cdots,n_m})$, but the same was obtained in~\cite[Corollary 2.5]{Zhou1} as a consequence of~\cite[Theorem 2.4]{Zhou1}, which can be used  for a wider class of graphs. Thus we state the result for the determinant of $D(K_{n_1,n_2,\cdots,n_m})$ without proof.

\begin{theorem}\label{thm:detD_single}
Let $D(K_{n_1,n_2,\cdots,n_m})$ be the distance matrix of complete $m$-partite graph $K_{n_1,n_2,\cdots,n_m}$ on $|V|=\sum_{i=1}^m n_i$ vertices. Then the determinant of the distance matrix is given by
$$\det D(K_{n_1,n_2,\cdots,n_m}) = (-2)^{|V|-m}\left[ \sum_{i=1}^m \left( n_i\prod_{j \neq i}(n_j - 2) \right) +  \prod_{i=1}^m (n_i - 2)\right].$$ 
\end{theorem}

Before proceeding further we first  prove a simple lemma  and introduce few notations useful for the subsequent results.
\begin{lem}\label{lem:sum}
Let $p$ be a positive integer, and let $r$ be an integer with $1 \le r \le 2p$. Then the equation
$$\sum_{i=1}^p  \frac{1}{q_i} = \frac{r}{2}$$
in $q_1,q_2,\cdots,q_p$ has a positive integer solution.
\end{lem}
\begin{proof}
Let $r$ be even. If $r = 2k$ for some  integer $k\geq 1$, then   choosing $$q_i=\begin{cases}
 1 & \mbox{ if } 1\leq i\leq k-1;\\
 p+1-k & \mbox{ if } k\leq i \leq p,
\end{cases}$$
we have $\ds \sum_{i=1}^p  \frac{1}{q_i}=  (k-1) + \dfrac{p-(k-1)}{p-(k-1)} =k=\dfrac{r}{2}.$  

Let $r$ be odd. If r=1, then choosing  $q_i=2p$ for $1\leq i\leq p$, we have $\ds \sum_{i=1}^p  \frac{1}{q_i} = \frac{1}{2}$. Next, if  $r = 2k+1$ for some  integer $k\geq 1$, then choosing 
$$q_i=\begin{cases}
 1 & \mbox{ if } 1\leq i\leq k;\\
 2(p-k) & \mbox{ if } k+1\leq i \leq p,
\end{cases}$$
we have $\ds \sum_{i=1}^p  \frac{1}{q_i}=  k + \dfrac{p-k}{ 2(p-k)} =k+\frac{1}{2}=\frac{2k+1}{2}=\dfrac{r}{2}.$ This completes the proof.
\end{proof}

Let $n_i\in \mathbb{N}$, $1\leq i\leq m$ and  $m\geq 2$, let us denote 
\begin{equation}\label{eqn:beta}
\begin{cases}
\ds \beta_{n_1, n_2,\cdots,n_m} = \sum_{i=1}^m  n_i\prod_{j \neq i}(n_j - 2)  +  \prod_{i=1}^m (n_i - 2),\\
\beta_{\widehat{n_i}}= \beta_{n_1, n_2,\cdots, n_{i-1},n_{i+1}\cdots,n_m},
\end{cases}
\end{equation}
and 
\begin{equation}\label{eqn:gamma}
\hspace*{-2.2cm}
\begin{cases}

\ds \gamma_{n_1, n_2,\cdots,n_m} = \sum_{i=1}^m  n_i\prod_{j \neq i}(n_j - 2),\\

\gamma_{\widehat{n_i}} =\gamma_{n_1, n_2,\cdots, n_{i-1},n_{i+1},\cdots,n_m}. 
\end{cases}
\end{equation}

 Now we are interested in the cases in which the determinant and cofactor of the distance matrix of $G=K_{n_1,n_2,\cdots,n_m}$ are zero. Unlike the case of complete bipartite graphs,   for complete $m$-partite graphs (with $m\geq 3$) the determinant and  cofactor vanishes for infinitely many partitions. A similar result for the determinant was obtained in~\cite[Corollary 2.5]{Zhou1}, but the result below provides some additional information useful for the subsequent section on multi-block graphs.

\begin{theorem}\label{thm:beta-l-1}
Let  $m\geq 2$ and $G=K_{n_1,n_2,\cdots,n_m}$. Then, $\det D(G)=0$ if and only if   either of the following holds:
\begin{enumerate}
\item[$(1)$] at least two $n_i$'s are $2$ for  $1\leq i\leq m$,

\item[$(2)$]  there exists   $l \in \mathbb{N}$ with $\ds \frac{m+1}{2} < l \leq \frac{3m+1}{4}$  such that  
$n_i=1$ for  $1\leq i\leq l$ and  $n_i > 2$ for $l+1\leq i\leq m$  with $$\ds 2\sum_{i=l+1}^m  \frac{1}{n_i-2} = 2l-(m+1).$$
\end{enumerate}
\end{theorem}
\begin{proof}
In view of Theorem~\ref{thm:detD_single}, since the result for case $(1)$ is easily  verifiable, so we proceed with the  case $(2)$. Since $\det D(G)$ is nonzero if exactly one of $n_i$ is $2$, so let us assume $n_i\neq 2$ for all $1\leq i\leq m$.  Then, $\det D(G)$ is zero only if  some $n_i$'s are $1$.  If $n_i =1$ for all $ i=1,2,\ldots, m$, then  $G=K_m$ and  $\det D(G)\neq 0.$ Next, for $l< m$, let us assume $n_i=1$ for  $1\leq i\leq l$ and $n_i>2$ for  $l+1\leq i\leq m$. Using Eqn.~\eqref{eqn:beta}, we get 
\begin{equation}\label{eqn:beta-l-1}
\beta_{n_1, n_2,\cdots,n_m} = (-1)^l \prod_{i=l+1}^m (n_i - 2) \left[2\sum_{i=l+1}^m  \frac{1}{n_i-2} + (m+1-2l)\right].
\end{equation}
Thus, by Theorem~\ref{thm:detD_single} we have  $\det D(G)=0$ if and only if  $\ds 2\sum_{i=l+1}^m  \frac{1}{n_i-2}   = 2l-(m+1)$. Using $n_i>2$ for  $l+1\leq i\leq m$, we get $\ds \frac{m+1}{2} < l \leq \frac{3m+1}{4}.$ Conversely, by  Lemma~\ref{lem:sum}, the sum
$$\ds 2\sum_{i=l+1}^m  \frac{1}{n_i-2} = 2l-(m+1), \mbox{ whenever  }\frac{m+1}{2} < l \leq \frac{3m+1}{4},$$
admit  integer solutions with  $n_i>2$   and using  Eqn.~\eqref{eqn:beta-l-1} the result follows.
\end{proof}

The next result gives a necessary and sufficient condition for which the cofactor of $D(K_{n_1,n_2,\cdots,n_m})$ is zero.
\begin{theorem}\label{thm:gamma-l-1}
Let  $m\geq 2$ and $G=K_{n_1,n_2,\cdots,n_m}$. Then, $\cof D(G)=0$ if and only if either of the following holds:  
\begin{enumerate}
\item[$(1)$] at least two $n_i$'s are $2$ for $1\leq i\leq m$, 

\item[$(2)$]   there exists   $l \in \mathbb{N}$ with $\ds  \frac{m}{2} < l \leq \frac{3m}{4}$  such that  
$n_i=1$ for  $1\leq i\leq l$ and  $n_i > 2$ for  $l+1\leq i\leq m$  with
$$\ds 2\sum_{i=l+1}^m  \frac{1}{n_i-2} = 2l-m.$$
\end{enumerate}
\end{theorem}
\begin{proof}
In view of Theorem~\ref{thm:cof} and  by arguing similar to Theorem~\ref{thm:beta-l-1},  it is enough to consider the  case $(2)$. Let $n_i=1$ for  $1\leq i\leq l$ and $n_i>2$ for  $l+1\leq i\leq m$. Using Eqn.~\eqref{eqn:gamma}, we get 
\begin{equation}\label{eqn:gamma-l-1}
\gamma_{n_1, n_2,\cdots,n_m} = (-1)^l \prod_{i=l+1}^m (n_i - 2) \left[2\sum_{i=l+1}^m \frac{1}{n_i-2}   + (m-2l)\right].
\end{equation}
Thus, by Theorem~\ref{thm:cof} we have  $\cof D(G)=0$ if and only if  $\ds 2\sum_{i=l+1}^m  \frac{1}{n_i-2}   = 2l-m$. In view of Eqn.~\eqref{eqn:gamma-l-1}, proceeding similar to the proof of Theorem~\ref{thm:beta-l-1} yields the result.
\end{proof}

In the next section, we  find the inverse of the distance matrix of complete $m$-partite graph $K_{n_1,n_2,\cdots,n_m}$ whenever it exists.

\section{Inverse of $D(K_{n_1,n_2,\cdots,n_m})$} \label{sec:inverse-Single-m-partitite}

 We first prove a result which gives a recurrence  type of relation involving $\beta$, $\gamma$  as defined in Eqns.~\eqref{eqn:beta}, \eqref{eqn:gamma}, respectively  and discuss some of its  consequences.
 
\begin{lem}\label{lem:index}
 For a given index set $\mathcal{I}=\{n_1,n_2,\cdots,n_m\}$,  let us denote  
$ \beta_\mathcal{I}=\beta_{n_1, n_2,\cdots,n_m},\  \gamma_\mathcal{I}=\gamma_{n_1, n_2,\cdots,n_m},\  \alpha_\mathcal{I}=\alpha_{n_1, n_2,\cdots,n_m}=\prod_{i=1}^{m}(n_i-2).$
\begin{enumerate}
\item [(i)] Let $\mathcal{I}=\{n_1,n_2,\cdots,n_m\}$,  if $\mathcal{J}  \subsetneq \mathcal{I}$ and $\mathcal{J}^{c}=\mathcal{I}\setminus \mathcal{J}$, then $\beta_\mathcal{I}=\gamma_\mathcal{I}+\alpha_\mathcal{I}$ and 
\begin{equation}\label{eqn:rec}
\begin{cases}
 \gamma_{\mathcal{I}} = \alpha_{\mathcal{J}^{c}} \gamma_{\mathcal{J}} + \alpha_{\mathcal{J}} \gamma_{\mathcal{J}^{c}}, \\
\beta_{\mathcal{I}} = \alpha_{\mathcal{J}^{c}} \beta_{\mathcal{J}} + \alpha_{\mathcal{J}} \gamma_{\mathcal{J}^{c}}. 
\end{cases}
\end{equation}

\item [(ii)]Let $\mathcal{I}=\{n_1,n_2,\cdots,n_m\}$, if $\beta_\mathcal{I}\neq 0$, then there exist a $\mathcal{J}  \subsetneq \mathcal{I}$ with $|\mathcal{J} |=m-1$ such that $\beta_\mathcal{J}\neq 0$. 
 \end{enumerate} 
\end{lem}
\begin{proof}
Observe, $\beta_\mathcal{I}=\gamma_\mathcal{I}+\alpha_\mathcal{I}$ follows from the definition. Next,  $$ \ds \gamma_{\mathcal{I}}  = \sum_{i=1}^m  n_i\prod_{j \neq i}(n_j - 2)  = \left( \sum_{\mathcal{J}} +\sum_{\mathcal{J}^{c}} \right) n_i \prod_{j \neq i}(n_j - 2)
 =\alpha_{\mathcal{J}^{c}} \gamma_{\mathcal{J}} + \alpha_{\mathcal{J}} \gamma_{\mathcal{J}^{c}}.$$
Further, using $\beta_\mathcal{I}=\gamma_\mathcal{I}+\alpha_\mathcal{I}$ and $\alpha_{\mathcal{I}}=\alpha_{\mathcal{J}}\alpha_{\mathcal{J}^c}$, we have 
 $$\beta_\mathcal{I}=\alpha_{\mathcal{J}^{c}} \gamma_{\mathcal{J}} + \alpha_{\mathcal{J}} \gamma_{\mathcal{J}^{c}}+ \alpha_{\mathcal{J}}\alpha_{\mathcal{J}^c}= \alpha_{\mathcal{J}^c}(\gamma_\mathcal{J}+\alpha_\mathcal{J})+ \alpha_{\mathcal{J}} \gamma_{\mathcal{J}^{c}}=\alpha_{\mathcal{J}^c}\beta_\mathcal{J}+ \alpha_{\mathcal{J}} \gamma_{\mathcal{J}^{c}}.$$ 
To prove part $(ii)$, suppose on the contrary  we assume no such $\mathcal{J}$ exists.  Then, by part $(i)$,  $\ds \beta_\mathcal{I}= \alpha_{\mathcal{J}} \gamma_{\mathcal{J}^{c}}=n_i\prod_{j=1\atop j\neq i}^{m} (n_j-2)  \mbox{ for all }i=1,2,\ldots,m.$ Therefore, for any $1\leq r,s\leq m$, we have $$\ds n_r\prod_{j=1\atop j\neq r}^{m} (n_j-2)= n_s\prod_{j=1\atop j\neq s}^{m} (n_j-2),$$ which implies  $n_r=n_s$ and hence $n_1=n_2=\cdots=n_m$. Since we assume  $\beta_{\mathcal{J}}=0$, whenever $|\mathcal{J}|=m-1$, so  by Theorem~\ref{thm:beta-l-1} we have $n_i=2$ for all $1\leq i\leq m$. This contradicts the hypothesis $\beta_\mathcal{I}\neq 0$ and hence the result follows.
\end{proof}

Following the convention  used in Eqns.~\eqref{eqn:beta} and~\eqref{eqn:gamma} we introduce a few notations. Let $\mathcal{I}=\{n_1,n_2,\cdots,n_m\}$ be an index set and   $\mathcal{J} = \mathcal{I}\setminus \{n_i,n_j\}$ for $1 \leq i,j \leq m$.  We denote
$$\alpha_{\widehat{n_i,n_j}} = \alpha_{\mathcal{J}}, \  \beta_{\widehat{n_i,n_j}} = \beta_{\mathcal{J}} \mbox{ and } \gamma_{\widehat{n_i,n_j}} = \gamma_{\mathcal{J}},$$ 
where $\alpha_{\mathcal{J}}, \ \beta_{\mathcal{J}}$ and $\gamma_{\mathcal{J}}$ are as defined in Lemma~\ref{lem:index}. Now we prove a few identities which will be used in computing the inverse of $D(K_{n_1,n_2,\cdots,n_m})$.

\begin{lem}\label{lem:identity}
For $2\leq i\leq m$, we have the following identities:

\begin{itemize}
\item[ (a)] 
$n_1 (\gamma_{\widehat{n_1}} + 2\alpha_{\widehat{n_1}}) - \beta_{n_1, n_2,\cdots,n_m}=2 \beta_{\widehat{n_1}}.$
\item[(b)] $n_i(\gamma_{\widehat{n_1,n_i}} + 2\alpha_{\widehat{n_1,n_i}}) - \beta_{\widehat{n_1}} - 2\gamma_{\widehat{n_1,n_i}}= 2\alpha_{\widehat{n_1,n_i}}.$

\item[(c)]$(\gamma_{\widehat{n_1,n_i}} +2 \alpha_{\widehat{n_1,n_i}})\beta_{n_1, n_2,\cdots,n_m} + 2n_1 \alpha_{\widehat{n_1,n_i}}^2= [2 \beta_{\widehat{n_i}}- \gamma_{\widehat{n_i}}]\beta_{\widehat{n_1}}.$
\end{itemize}
\end{lem}

\begin{proof}
Using the part $(i)$ of Lemma~\ref{lem:identity}, we have
\begin{align*}
n_1 (\gamma_{\widehat{n_1}} + 2\alpha_{\widehat{n_1}}) - \beta_{n_1, n_2,\cdots,n_m} & = n_1 (\gamma_{\widehat{n_1}} + 2\alpha_{\widehat{n_1}}) - \left(\alpha_{\widehat{n_1}} \beta_{n_1} + \alpha_{n_1}\gamma_{\widehat{n_1}} \right)\\
			   & = n_1 (\gamma_{\widehat{n_1}} + 2\alpha_{\widehat{n_1}}) - \left(2(n_1-1)\alpha_{\widehat{n_1}} + (n_1-2)\gamma_{\widehat{n_1}} \right)\\ & = 2 (\gamma_{\widehat{n_1}} + \alpha_{\widehat{n_1}}) = 2 \beta_{\widehat{n_1}}.
\end{align*}
This proves part $(a)$. Next using the index set $\mathcal{I'} = \{n_2, n_3, \cdots,n_m\}$ and $\mathcal{J'} = \mathcal{I'} \setminus \{n_i\}$, whenever $2 \leq i \leq m$, by part $(a)$ we get $n_i(\gamma_{\widehat{n_1,n_i}} + 2\alpha_{\widehat{n_1,n_i}}) - \beta_{\widehat{n_1}} = 2\beta_{\widehat{n_1,n_i}}.$ Since $\beta_{\widehat{n_1,n_i}} = \gamma_{\widehat{n_1,n_i}}+\alpha_{\widehat{n_1,n_i}}$, part $(b)$ follows. Finally with the repeated application of Eqn.~\eqref{eqn:rec} we prove part $(c)$ and the proof is as follows:
\begin{align*}
  &(\gamma_{\widehat{n_1,n_i}} +2 \alpha_{\widehat{n_1,n_i}})\beta_{n_1, n_2,\cdots,n_m} + 2n_1 \alpha_{\widehat{n_1,n_i}}^2\\
 = &(\gamma_{\widehat{n_1,n_i}} +2 \alpha_{\widehat{n_1,n_i}})(\alpha_{n_1}\beta_{\widehat{n_1}} + \alpha_{\widehat{n_1}} \gamma_{n_1}) + 2n_1 \alpha_{\widehat{n_1,n_i}}^2 \\
 = &(\gamma_{\widehat{n_1,n_i}} +2 \alpha_{\widehat{n_1,n_i}})\alpha_{n_1}\beta_{\widehat{n_1}} + n_1 \alpha_{\widehat{n_1}} (\gamma_{\widehat{n_1,n_i}} +2 \alpha_{\widehat{n_1,n_i}}) + 2n_1 \alpha_{\widehat{n_1,n_i}}^2\\
 = &(\gamma_{\widehat{n_1,n_i}} +2 \alpha_{\widehat{n_1,n_i}})\alpha_{n_1}\beta_{\widehat{n_1}} + n_1 \alpha_{\widehat{n_1}} \gamma_{\widehat{n_1,n_i}} +2n_1(n_i-2) \alpha_{\widehat{n_1,n_i}}^2 + 2n_1 \alpha_{\widehat{n_1,n_i}}^2\\
 = &(\gamma_{\widehat{n_1,n_i}} +2 \alpha_{\widehat{n_1,n_i}})\alpha_{n_1}\beta_{\widehat{n_1}} + n_1 \alpha_{\widehat{n_1}} \gamma_{\widehat{n_1,n_i}} +2n_1(n_i-1) \alpha_{\widehat{n_1,n_i}}^2\\
 = &(\gamma_{\widehat{n_1,n_i}} +2 \alpha_{\widehat{n_1,n_i}})\alpha_{n_1}\beta_{\widehat{n_1}} + n_1(n_i-2) \alpha_{\widehat{n_1,n_i}} \gamma_{\widehat{n_1,n_i}} +2n_1(n_i-1) \alpha_{\widehat{n_1,n_i}}^2\\
 = &(\gamma_{\widehat{n_1,n_i}} +2 \alpha_{\widehat{n_1,n_i}})\alpha_{n_1}\beta_{\widehat{n_1}} + n_1\alpha_{\widehat{n_1,n_i}} \left[ (n_i-2)\gamma_{\widehat{n_1,n_i}} +2(n_i-1) \alpha_{\widehat{n_1,n_i}} \right]\\
 = &(\gamma_{\widehat{n_1,n_i}} +2 \alpha_{\widehat{n_1,n_i}})\alpha_{n_1}\beta_{\widehat{n_1}} + n_1\alpha_{\widehat{n_1,n_i}} \left[ ((n_i-2)\gamma_{\widehat{n_1,n_i}} +n_i \alpha_{\widehat{n_1,n_i}}) +\alpha_{\widehat{n_1}} \right]\\
  = &(\gamma_{\widehat{n_1,n_i}} +2 \alpha_{\widehat{n_1,n_i}})\alpha_{n_1}\beta_{\widehat{n_1}} + n_1\alpha_{\widehat{n_1,n_i}} \left[ \gamma_{\widehat{n_1}}   +\alpha_{\widehat{n_1}} \right]\\
 = &(\gamma_{\widehat{n_1,n_i}} +2 \alpha_{\widehat{n_1,n_i}})\alpha_{n_1}\beta_{\widehat{n_1}} + n_1\alpha_{\widehat{n_1,n_i}}\beta_{\widehat{n_1}}\\
 = &[(\gamma_{\widehat{n_1,n_i}} +2 \alpha_{\widehat{n_1,n_i}})(n_1-2) + n_1\alpha_{\widehat{n_1,n_i}}]\beta_{\widehat{n_1}}\\
 = &[(n_1-2)\gamma_{\widehat{n_1,n_i}} + n_1\alpha_{\widehat{n_1,n_i}} +2 \alpha_{\widehat{n_i}} ]\beta_{\widehat{n_1}}\\
 = & [\gamma_{\widehat{n_i}} +2 \alpha_{\widehat{n_i}}]\beta_{\widehat{n_1}} =[2 \beta_{\widehat{n_i}}- \gamma_{\widehat{n_i}}]\beta_{\widehat{n_1}}. 
\end{align*}

\end{proof}

\begin{theorem}
Let  $D(K_{n_1,n_2,\cdots,n_m})$ be the distance matrix of complete $m$-partite graph $K_{n_1,n_2,\cdots,n_m}$. If $\det D(K_{n_1,n_2,\cdots,n_m}) \neq 0$, then  the inverse in $m\times m$ block form is given by   $D(K_{n_1,n_2,\cdots,n_m})^{-1} = [\widetilde{D}_{ij}]$, where 
\begin{equation*}\label{eqn:D_inv}
\widetilde{D}_{ij}=
\begin{cases}
\left(\dfrac{2 \beta_{\widehat{n_i}} - \gamma_{\widehat{n_i}}}{2 \beta_{n_1, n_2,\cdots,n_m}} \right)J_{n_i} - \dfrac{1}{2} I_{n_i} & \text{if} \ i = j;\\
\\
-\dfrac{\displaystyle{\prod_{l \neq i,j}(n_l - 2)}}{\beta_{n_1, n_2,\cdots,n_m}} J_{n_i \times n_j} & \text{if} \ i \neq j.\\
\end{cases}
\end{equation*}
\end{theorem}
\begin{proof}
We prove this result using induction on $m$ and Proposition~\ref{prop:schur}. By~\cite[Lemma 2]{Hou3}, the inverse of the distance matrix for bipartite graph is given by
{\small \begin{equation*}
D(K_{n_1,n_2})^{-1}=\left[
\begin{array}{c|c}
 \dfrac{3n_2-4}{2(3n_1n_2-4(n_2+n_2-1))}J_{n_1}-\dfrac{1}{2}I_{n_1} & -\dfrac{1}{3n_1n_2-4(n_1+n_2-1)}J_{n_1\times n_2} \\
  \midrule
-\dfrac{1}{3n_1n_2-4(n_1+n_2-1)}J_{n_2\times n_1} &  \dfrac{3n_1-4}{2(3n_1n_2-4(n_1+n_2-1))}J_{n_2}-\dfrac{1}{2}I_{n_2}
\end{array}
\right],
\end{equation*}}
and it is easy to check the result is true for $m=2$.  We assume the result is true for complete $(m-1)$-partite graphs. The distance matrix of complete $m$-partite graph can be expressed in the following block form: \\
$$D(K_{n_1,n_2,\cdots,n_m}) = \left[
\begin{array}{c|c}
B_{11}& B_{12} \\
\hline
B_{21} & B_{22}
\end{array}
\right],$$
where 

$B_{11} = 2(J_{n_1} - I_{n_1})$, \ $B_{12} = \left[
\begin{array}{c|c|c|c}
J_{n_1 \times n_2}& J_{n_1 \times n_3} & \cdots &J_{n_1 \times n_m} \\
\end{array}
\right] = B_{21}^t$ and $B_{22} = D(K_{n_2,n_3,\cdots,n_m})$. 

By  part $(ii)$ of Lemma~\ref{lem:index}, without loss of generality we assume $B_{22} = D(K_{n_2,n_3,\cdots,n_m})$ is invertible.  Using the induction hypothesis for $B_{22}$, the inverse $B_{22}^{-1}$ is an $(m-1)\times (m-1)$ block matrix. Since the vertex set of $K_{n_2,n_3,\cdots,n_m}$ corresponds to the partitions $V_2,V_3,\ldots,V_m$, so without loss of generality, we index the  block matrix   $B_{22}^{-1}$ from $2,3,\ldots,m$,  instead of $1,2,\ldots,m-1$.
By induction hypothesis, let $B_{22}^{-1}=[\widetilde{(B_{22})}_{ij}]$, for $2\leq i,j\leq m$, where 
\begin{equation*}
\begin{split}
\widetilde{(B_{22})}_{ij} & = \begin{cases}\vspace*{.4cm}
  \dfrac{2\beta_{\widehat{n_1,n_i} } - \gamma_{\widehat{n_1,n_i} }}{2 \beta_{\widehat{n_1}}} J_{n_i} - \dfrac{1}{2} I_{n_i} & \text{if} \ i = j;\\  
  -\dfrac{\prod_{l \neq 1,i,j}(n_l - 2)}{\beta_{\widehat{n_1}}} J_{n_i \times n_j} & \text{if} \ i \neq j,\\
  \end{cases}
\end{split}
\end{equation*}
with $\beta_{\widehat{n_1,n_i} }=\beta_{n_2,\cdots,n_{i-1},n_{i+1},\cdots,n_m}$ and $\gamma_{\widehat{n_1,n_i} }=\gamma_{n_2,\cdots, n_{i-1},n_{i+1},\cdots,n_m}$. Since $B_{22}^{-1}B_{21}$ is an $(m-1)\times 1$ block matrix, so using part $(b)$ of Lemma~\ref{lem:identity} and $\alpha_{\widehat{n_1,n_i} }=\alpha_{n_2,\cdots,n_{i-1},n_{i+1},\cdots,n_m}$; for $2\leq i\leq m$, we have 
\begin{align*}
(B_{22}^{-1}B_{21})_i
 & = \left(\dfrac{2 \beta_{\widehat{n_1,n_i}} - \gamma_{\widehat{n_1,n_i}}}{2 \beta_{\widehat{n_1}}} J_{n_i} - \dfrac{1}{2} I_{n_i}  \right)J_{n_i \times n_1} -\sum_{j=2,j \neq i}^m \dfrac{\prod_{l \neq 1,i,j}(n_l - 2)}{\beta_{\widehat{n_1}}} J_{n_i \times n_j} J_{n_j \times n_1} \\
  & = \dfrac{1}{2 \beta_{\widehat{n_1}}} \left[n_i(2 \beta_{\widehat{n_1,n_i}} - \gamma_{\widehat{n_1,n_i}}) - \beta_{\widehat{n_1}} -2\sum_{j=2,j \neq i}^m  n_j \prod_{l \neq 1,i,j}(n_l - 2)\right]J_{n_i \times n_1}\\
  & = \dfrac{1}{2 \beta_{\widehat{n_1}}} \left[ n_i(\gamma_{\widehat{n_1,n_i}} + 2\alpha_{\widehat{n_1,n_i}}) - \beta_{\widehat{n_1}} - 2\gamma_{\widehat{n_1,n_i}} \right]J_{n_i \times n_1}\\
  & = \frac{\alpha_{\widehat{n_1,n_i}}}{\beta_{\widehat{n_1}}} J_{n_i \times n_1}= \frac{\prod_{l \neq 1,i}(n_l - 2)}{\beta_{\widehat{n_1}}} J_{n_i \times n_1}.
\end{align*}
Thus
\begin{equation}\label{eqn:B22B12}
B_{22}^{-1}B_{21} =  \left[
\begin{array}{c|c|c|c}
\frac{\prod_{l \neq 1,2}(n_l - 2)}{\beta_{\widehat{n_1}}}J_{n_2 \times n_1}&\frac{\prod_{l \neq 1,3}(n_l - 2)}{\beta_{\widehat{n_1}}} J_{n_3 \times n_1} & \cdots &\frac{\prod_{l \neq 1,m}(n_l - 2)}{\beta_{\widehat{n_1}}} J_{n_m \times n_1} \\
\end{array}
\right]^t,
\end{equation}
and hence $\ds B_{12}B_{22}^{-1}B_{21} = \dfrac{\sum_{j=2}^m n_j \alpha_{\widehat{n_1,n_j}}}{\beta_{\widehat{n_1}}} J_{n_1} = \dfrac{\gamma_{\widehat{n_1}}}{\beta_{\widehat{n_1}}}J_{n_1}$.
Then, the Schur complement of $B_{22}$ in $D(K_{n_1,n_2,\cdots,n_m})$ is 
$$D(K_{n_1,n_2,\cdots,n_m}) /B_{22}=B_{11} -B_{12}B_{22}^{-1}B_{21}= \left(\dfrac{2 \beta_{\widehat{n_1}} - \gamma_{\widehat{n_1}}}{\beta_{\widehat{n_1}}} \right)J_{n_1} - 2I_{n_1},$$
and using Lemma~\ref{lem:aI+bJ}, we get
$$(D(K_{n_1,n_2,\cdots,n_m}) /B_{22})^{-1} = \left(\dfrac{2 \beta_{\widehat{n_1}} - \gamma_{\widehat{n_1}}}{2 \beta_{n_1, n_2,\cdots,n_m}} \right)J_{n_1} - \dfrac{1}{2} I_{n_1}.$$
Let $D(K_{n_1,n_2,\cdots,n_m})^{-1} =\left[
\begin{array}{c|c}
C_{11} & C_{12}  \\
\hline
C_{21} & C_{22}
\end{array}
\right]$ be the conformal partition as in Proposition~\ref{prop:schur},  with respect to the Schur complement of $B_{22}$ in $D(K_{n_1,n_2,\cdots,n_m})$. Using part $(a)$ of Lemma~\ref{lem:identity}, for  $2\leq i\leq m$, we have 
\begin{align*}
\left[ \left(\dfrac{2 \beta_{\widehat{n_1}} - \gamma_{\widehat{n_1}}}{2 \beta_{n_1, n_2,\cdots,n_m}} \right)J_{n_1} - \dfrac{1}{2} I_{n_1} \right] \times \frac{\alpha_{\widehat{n_1,n_i}}}{\beta_{\widehat{n_1}}}J_{n_1 \times n_i}
& = \left[ n_1 \left(\dfrac{2 \beta_{\widehat{n_1}} - \gamma_{\widehat{n_1}}}{2 \beta_{n_1, n_2,\cdots,n_m}} \right) - \dfrac{1}{2} \right] \times \frac{\alpha_{\widehat{n_1,n_i}}}{\beta_{\widehat{n_1}}}J_{n_1 \times n_i}\\
& = \left[\dfrac{ n_1 \left(2 \beta_{\widehat{n_1}} - \gamma_{\widehat{n_1}} \right) - \beta_{n_1, n_2,\cdots,n_m}}{2 \beta_{n_1, n_2,\cdots,n_m}}  \right] \times \frac{\alpha_{\widehat{n_1,n_i}}}{\beta_{\widehat{n_1}}}J_{n_1 \times n_i}\\
&= \frac{\alpha_{\widehat{n_1,n_i}}}{\beta_{n_1, n_2,\cdots,n_m}}J_{n_1 \times n_i}= \frac{\prod_{l \neq 1,i}(n_l - 2)}{\beta_{n_1, n_2,\cdots,n_m}}J_{n_1 \times n_i}.
\end{align*}
Since $B_{12}^t=B_{21}$ and $B_{22}^{-1}$ is a symmetric matrix, we have $B_{12}B_{22}^{-1}=(B_{22}^{-1} B_{21})^t$, so  the above calculations yields:
\begin{equation}\label{eqn:C12}
\ds\begin{split}
C_{12} & = -(D(K_{n_1,n_2,\cdots,n_m}) /B_{22})^{-1}B_{12}B_{22}^{-1}\\
	   & = -\left[
\begin{array}{c|c|c|c}
\frac{\prod_{l \neq 1,2}(n_l - 2)}{\beta_{n_1, n_2,\cdots,n_m}}J_{n_1 \times n_2}&\frac{\prod_{l \neq 1,3}(n_l - 2)}{\beta_{n_1, n_2,\cdots,n_m}} J_{n_1 \times n_3} & \cdots &\frac{\prod_{l \neq 1,m}(n_l - 2)}{\beta_{n_1, n_2,\cdots,n_m}} J_{n_1 \times n_m} \\
\end{array}
\right].
\end{split}
\end{equation}
By symmetry it can be seen that $C_{21} = C_{12}^t$. Next, let us denote $$X=B_{22}^{-1}B_{21}(D(K_{n_1,n_2,\cdots,n_m}) /B_{22})^{-1}B_{12}B_{22}^{-1}.$$ 
By Eqns.~\eqref{eqn:B22B12} and \eqref{eqn:C12}, it is easy to notice that  $X=B_{22}^{-1}B_{21}(-C_{12})$ is an $(m-1)\times (m-1)$ block matrix and similar to $B_{22}^{-1}$ we index the block form $X=[X_{ij}]$;  $2\leq i,j\leq m$, where
\begin{equation*}
X_{ij} = \begin{cases}\vspace*{.4cm}
  \dfrac{n_1 \prod_{l \neq 1,i}(n_l - 2)^2}{\beta_{\widehat{n_1}}  \beta_{n_1, n_2,\cdots,n_m}} J_{n_i} & \text{if} \ i = j;\\
  \dfrac{n_1 (n_i -2)(n_j-2) \prod_{l \neq 1,i,j}(n_l - 2)^2}{\beta_{\widehat{n_1}}  \beta_{n_1, n_2,\cdots,n_m}} J_{n_i \times n_j} & \text{if} \ i \neq j.\\
  \end{cases}
\end{equation*} 
Let $C_{22}= [(C_{22})_{ij}]$  be the block form of $C_{22}$. Since $C_{22}=B_{22}^{-1}+X$, so for $2\leq i,j\leq m$, we have
\begin{align*}
(C_{22})_{ij}= \widetilde{(B_{22})}_{ij}+ X_{ij}&=  \begin{cases}\vspace*{.4cm}
  \left(\dfrac{2\beta_{\widehat{n_1,n_i}} - \gamma_{\widehat{n_1,n_i}}}{2 \beta_{\widehat{n_1}}} + \dfrac{n_1 \prod_{l \neq 1,i}(n_l - 2)^2}{\beta_{\widehat{n_1}}  \beta_{n_1, n_2,\cdots,n_m}} \right)J_{n_i} - \dfrac{1}{2} I_{n_i} & \text{if} \ i = j;\\
  \left(\dfrac{\prod_{l \neq 1,i,j}(n_l - 2)}{\beta_{\widehat{n_1}}} \right) \left(\dfrac{n_1 \prod_{l \neq 1}(n_l - 2)}{\beta_{n_1, n_2,\cdots,n_m}} - 1 \right)J_{n_i \times n_j} & \text{if} \ i \neq j.\\
  \end{cases}
\end{align*}
Now using the Lemmas~\ref{lem:index} and~\ref{lem:identity},  we compute the following:
\begin{align*}
\dfrac{2\beta_{\widehat{n_1,n_i}} - \gamma_{\widehat{n_1,n_i}}}{2 \beta_{\widehat{n_1}}} + \dfrac{n_1 \prod_{l \neq 1,i}(n_l - 2)^2}{\beta_{\widehat{n_1}}  \beta_{n_1, n_2,\cdots,n_m}}
& = \dfrac{(2 \beta_{\widehat{n_1,n_i}} - \gamma_{\widehat{n_1,n_i}})\beta_{n_1, n_2,\cdots,n_m} + 2n_1 \prod_{l \neq 1,i}(n_l - 2)^2}{2 \beta_{\widehat{n_1}}\beta_{n_1, n_2,\cdots,n_m} }\\
& = \dfrac{(\gamma_{\widehat{n_1,n_i}} +2 \alpha_{\widehat{n_1,n_i}})\beta_{n_1, n_2,\cdots,n_m} + 2n_1 \alpha_{\widehat{n_1,n_i}}^2 }{2 \beta_{\widehat{n_1}}\beta_{n_1, n_2,\cdots,n_m} }\\
&=\dfrac{2 \beta_{\widehat{n_i}} - \gamma_{\widehat{n_i}}}{2 \beta_{n_1, n_2,\cdots,n_m}}
\end{align*}
and 
\begin{align*}
\left(\dfrac{\prod_{l \neq 1,i,j}(n_l - 2)}{\beta_{\widehat{n_1}}}\right)   \left(\dfrac{n_1 \prod_{l \neq 1}(n_l - 2)}{\beta_{n_1, n_2,\cdots,n_m}} - 1 \right)
& = \dfrac{\prod_{l \neq 1,i,j}(n_l - 2)}{\beta_{\widehat{n_1}}} \times \left(\dfrac{n_1 \alpha_{\widehat{n_1}} - \beta_{n_1, n_2,\cdots,n_m}}{\beta_{n_1, n_2,\cdots,n_m}} \right)\\
& = - \dfrac{\prod_{l \neq 1,i,j}(n_l - 2)}{\beta_{\widehat{n_1}}} \times \dfrac{(n_1-2) \beta_{\widehat{n_1}} }{\beta_{n_1, n_2,\cdots,n_m}}\\
& = -\dfrac{{\prod_{l \neq i,j}(n_l - 2)}}{\beta_{n_1, n_2,\cdots,n_m}}.
\end{align*}
Hence the result follows.
\end{proof}

\section{ Distance Matrix of Multi-block Graphs}\label{sec:multi-block}

In this section, we aim to compute the determinant and inverse for a multi-block graph $G$ subject to the condition that $\cof D(G)\neq 0$. Recall that, a connected graph is said to be a multi-block graph if each of its blocks is  a complete $m$-partite graph, whenever $m\geq 2$. We achieve our goal by proving  the results for single blocks and further extend these results for multi-block graphs using~\cite{Zhou1}.  Thus, we first define a few notions on single blocks and then extend these notions to multi-block graphs to prove the requisite results.

Let us begin with a  block of a multi-block graph. In this paragraph, we suppose $G$ is the complete $m$-partite graph $K_{n_1,n_2,\cdots,n_m}$ with $m \ge 2$  and $\cof D(G)\neq 0$. We can define
\begin{equation}\label{eqn:lambda_G-single}
\lambda_{G}= \frac{\det D(G)}{\cof D(G)} =\dfrac{\beta_{n_1,n_2,\cdots,n_m}}{\gamma_{n_1,n_2,\cdots,n_m}}.
\end{equation}
By Theorem~\ref{thm:gamma-l-1}, we know $\# \{i | n_i =2, 1 \le i \le m\} \le 1$. If $\# \{i | n_i =2, 1 \le i \le m\} = 1$, then $\cof D(G) = \det D(G) \ne 0$, and so $\lambda_G = 1 >0$. If $\# \{i | n_i =2, 1 \le i \le m\} = 0$ and $n_i >2$ for all $1 \le i \le m$, then $\det D(G) \ne 0$, $\cof D(G) \ne 0$ and $\lambda_G > 1 >0$. In the rest of the paragraph we may assume there exists $l < m$ such that $n_i = 1$ for $1 \le i \le l$ and $n_j > 2$ for $l < j \le m$. Then
$$
\gamma_{n_1,n_2,\cdots,n_m} = \left( \prod_{i=1}^m (n_i - 2) \right) \left( \sum_{i=1}^m \dfrac{n_i}{n_i-2} \right) = \left( \prod_{i=1}^m (n_i - 2) \right)  \left( (m-2l)+ 2\sum_{i=l+1}^m \dfrac{1}{n_i-2} \right).
$$
Now we consider the function
\begin{equation}
f(n_{l+1},\cdots,n_m) = \sum_{i=1}^m \dfrac{n_i}{n_i-2} = (m-2l)+2 \sum_{i=l+1}^m \frac{1}{n_i - 2}.
\end{equation}
Since $\cof D(G) \neq 0$, we have $\gamma_{n_1,\cdots,n_m} \ne 0$ and so $f(n_{l+1},\cdots,n_m) \ne 0$. Note that
$$\beta_{n_1,n_2,\cdots,n_m} =  \left( \prod_{i=1}^m (n_i - 2) \right) \left(1+ \sum_{i=1}^m \dfrac{n_i}{n_i-2}\right) = \left( \prod_{i=1}^m (n_i - 2) \right) \left(1+ (m-2l)+2 \sum_{i=l+1}^m \frac{1}{n_i - 2}\right)$$ and hence we have $$\lambda_G = 1 + \dfrac{1}{\ds \sum_{i=1}^m \dfrac{n_i}{n_i-2}} =1+ \dfrac{1}{f(n_{l+1},\cdots,n_m)}.$$ Thus $\lambda_G <0$ if and only if $-1 < f(n_{l+1},\cdots,n_m) < 0$ if and only if
\begin{equation}\label{eqn:eqn_2}
\dfrac{2l-m-1}{2} < \sum_{i=l+1}^m \frac{1}{n_i - 2} <\dfrac{2l-m}{2}.
\end{equation}
Note that $0 < \ds \sum_{i=l+1}^m \dfrac{1}{n_i - 2} \le m-l$. If the inequalities in Eqn~\eqref{eqn:eqn_2} has a solution then we have $0 < \dfrac{2l-m}{2}$ and $\dfrac{2l-m-1}{2} < m-l$, which implies
\begin{equation}\label{eqn:eqn_3}
\frac{m}{2} < l < \frac{3m+1}{4}.
\end{equation}

Now we give all possible complete $m$-partite graph $G$ with $\cof D(G)\neq 0$ satisfying $\lambda_G <0$ for $2\leq m \leq 5$.

\begin{prop}\label{prop:1}
Let $m \ge 2$ and $G$ be a complete $m$-partite graph  with $\cof D(G) \ne 0$. If  $G=K_{n_1,n_2,\cdots,n_m}$, and there exists $l < m$ such that $n_i = 1$ for $1 \le i \le l$ and $n_j > 2$ for $l < j \le m$, then the following results hold.
\begin{itemize}
\item[$(i)$] For $m = 2$ or $4$, there is no complete $m$-partite graph $G$ satisfying $\lambda_G < 0$.
\item[$(ii)$] Let $m = 3$. Then $\lambda_G < 0$ if and only if $G = K_{1,1,n_3}$ for some $n_3 \ge 5$.
\item[$(iii)$] Let $m=5$. Then $\lambda_G < 0$ if and only if $G = K_{1,1,1,n_4,n_5}$ for some $n_4 \ge 5$ and $n_5 > 4+\dfrac{4}{n_4-4}$.
\end{itemize}
\end{prop}

\begin{proof}
When $m = 2$, there is no $l$ satisfying inequalities in Eqn~\eqref{eqn:eqn_4}. When $m = 4$, by inequalities in Eqn~\eqref{eqn:eqn_4}, we have $l = 3$. So we have $G = K_{1,1,1,n_4}$ for some $n_4 > 2$. Then $f(n_4) = -2 + \dfrac{2}{n_4-2}$. We have $f(n_4) = 0$ when $n_4 = 3$ and $f(n_4) \le -1$ for $n_4 > 4$. There is no $n_4$ such that $-1 < f(n_4) < 0$. This proves part $(i)$.

To prove part $(ii)$, let $m = 3$. Then $l = 2$ by the inequalities in Eqn~\eqref{eqn:eqn_4}, $G = K_{1,1,n_3}$ for some $n_3 \ge 3$, and $f(n_3) = -1 + \dfrac{2}{n_3-2}$. We have $f(3) = 1, f(4) = 0$ and $-1 < f(n_3) < 0$ for $n_3 \ge 5$.

Finally, let $m=5$. Then $l=3$ by inequalities in Eqn~\eqref{eqn:eqn_4}, $G = K_{1,1,1,n_4,n_5}$ for some $n_4 \ge 3$ and $n_5 \ge 3$, and $f(n_4,n_5) = -1 + \dfrac{2}{n_4-2}+ \dfrac{2}{n_5-2}$. When $n_4=3$ we have $f(n_4,n_5) = \dfrac{2}{n_5-2} > 0$ for all $n_5
 \ge 3$. Suppose $n_4 \ge 5$, then $$
-1 < f(n_4,n_5) = -1 + \dfrac{2}{n_4-2}+ \dfrac{2}{n_5-2} < 0
$$ if and only if $n_5 > 4+\dfrac{4}{n_4-4}$. This completes the proof of part $(iii)$.
\end{proof}

Next, we show that similar to the case $m=5$, for $m \ge 6$ there are infinitely many $m$-partite graphs $G$ with $\cof D(G)\neq 0$ such that $\lambda_G < 0$. To establish this we first   prove the following   lemma.
\begin{lem}\label{lem:lem_1}
Let $p$ and $q$ be non-negative integers, and $n_i>2$ for all  $1 \le i \le p+q.$ Then the following results hold.
\begin{itemize}
\item[(i)] The inequality 
\begin{equation}\label{eqn:eqn_1_1}
p < \ds \sum_{i=1}^{p+q} \dfrac{1}{n_i - 2} < p+ \dfrac{1}{2}
\end{equation}
in $n_1,n_2,\ldots,n_{p+q}$ has a positive integer solution if and only if $q \ge 1$. Moreover, for each $q \ge 1$ the inequality in Eqn.~\eqref{eqn:eqn_1_1} has infinitely many solutions.
\item[(ii)] The inequality
\begin{equation}\label{eqn:eqn_1_2}
p + \dfrac{1}{2} <\ds \sum_{i=1}^{p+q} \dfrac{1}{n_i - 2} < p+1
\end{equation}
in $n_1,n_2,\ldots,n_{p+q}$ has a positive integer solution if and only if $q \ge 2$. Moreover, for each $q \ge 2$ the inequality in Eqn.~\eqref{eqn:eqn_1_2} has infinitely many solutions.
\end{itemize}
\end{lem}
\begin{proof}
Let $q\geq 1$. When $n_i=3$ for all $1\leq i\leq p$, the inequality   in Eqn.~\eqref{eqn:eqn_1_1} reduces to 
\begin{equation}\label{eqn:eqn_1_1_1}
0 < \ds \sum_{i=p+1}^{p+q} \dfrac{1}{n_i - 2} <  \dfrac{1}{2}.
\end{equation}
Thus,  $n_i>2(q+1)$ for all $p+1\leq i \leq p+q$, satisfies the inequality in Eqn.~\eqref{eqn:eqn_1_1_1} and hence the inequality in Eqn.~\eqref{eqn:eqn_1_1} has infinitely many solutions. Conversely, let $q=0$. Then  $0<\ds\sum_{i=1}^{p} \frac{1}{n_i - 2} \leq p$, whenever $n_i>2$  are  integers for all  $1 \le i \le p$, and hence the inequality in Eqn.~\eqref{eqn:eqn_1_1} has no positive integer solution. This proves part $(i)$. 

Next, let  $q\geq 2$. When $n_i=3$ for all $1\leq i\leq p$, the inequality in Eqn.~\eqref{eqn:eqn_1_2} reduces to
\begin{equation}\label{eqn:eqn_1_2_1}
\dfrac{1}{2} < \ds\sum_{i=p+1}^{p+q} \dfrac{1}{n_i - 2} <  1.
\end{equation}
Further, when $$n_i=\begin{cases}
4 & \mbox{ if } i=p+1;\\
n(q-1)+2 & \mbox{ if } p+2\leq i\leq p+q \mbox{ and for some }  n\geq 3,
\end{cases}$$
we have $ \ds \sum_{i=p+1}^{p+q} \dfrac{1}{n_i - 2}=\frac{1}{2}+\frac{q-1}{n(q-1)}=\frac{1}{2}+\frac{1}{n}$ and hence satisfies the inequality in Eqn.~\eqref{eqn:eqn_1_2_1}. Since $n\geq 3$, so the inequality in Eqn.~\eqref{eqn:eqn_1_2} has infinitely many solutions. Conversely, if $q=0$, then $0<\ds\sum_{i=1}^{p} \dfrac{1}{n_i - 2} \leq p$, whenever $n_i>2$  are  integers for all  $1 \le i \le p$. Further, if $q=1$, then $\ds\sum_{i=1}^{p} \dfrac{1}{n_i - 2}$ either equals to $p+1$ or less than equal to $p+\frac{1}{2}$, whenever $n_i>2$  are  integers for all  $1 \le i \le p+1$. Therefore, the inequality in Eqn.~\eqref{eqn:eqn_1_2} has no positive integer solution. This completes the proof.
\end{proof}

\begin{prop}\label{prop:2}
For $m \ge 5$,  there are infinitely many complete $m$-partite graph $G$  with $\cof D(G) \ne 0$ such that $\lambda_G < 0$.
\end{prop}
\begin{proof} 
Let $G=K_{n_1,n_2,\cdots,n_m}$, and there exists $l < m$ such that $n_i = 1$ for $1 \le i \le l$ and $n_j > 2$ for $l < j \le m$. Observe that, the inequality in Eqn~\eqref{eqn:eqn_3} can also be written as follows
\begin{equation}\label{eqn:eqn_4}
\begin{cases}
2x+1 \le l \le 3x & \text{ if } m=4x;\\
2x+1 \le l \le 3x & \text{ if } m=4x+1;\\
2x+2 \le l \le 3x+1 & \text{ if } m=4x+2;\\
2x+2 \le l \le 3x+2 & \text{ if } m=4x+3.\\
\end{cases}
\end{equation}
Thus we proceed for the proof by considering each of the above four cases  separately.

Let $m=4x$. Then by Eqn.~\eqref{eqn:eqn_4}, we have $l=2x+k$ for $1\leq k\leq x$, and hence the inequality in  Eqn~\eqref{eqn:eqn_2} reduces to
\begin{equation}\label{eqn:eqn_5}
(k-1) + \frac{1}{2} < \sum_{i=1}^{2x-k} \dfrac{1}{n_{2x+k+i} - 2} < k \ \mbox{ for } 1\leq k\leq x.
\end{equation}
By  Eqn.~\eqref{eqn:eqn_2}, we know that there is an one to one correspondence between integer $(\geq 3)$ solutions of the inequality in Eqn.~\eqref{eqn:eqn_5} and complete $m$-partite graphs $G$ with $\lambda_G<0$. Further, by part $(ii)$ of $Lemma~\ref{lem:lem_1}$, the inequality in Eqn.~\eqref{eqn:eqn_5} has infinitely many  integer ($\geq 3$) solutions   if and only if $(2x-k)-(k-1)\geq 2$ and $1\leq k\leq x.$ Hence there are infinitely many complete $m$-partite graph $G$ with $\lambda_G<0$  if and only if $1\leq k\leq \frac{2x-1}{2}.$ Since $m\geq 5$ and $m=4x$, so $\{k: 1\leq k\leq \frac{2x-1}{2}\}$ is a non-empty set and hence the result holds true for this case.

Let $m = 4x+1$. Then by Eqn.~\eqref{eqn:eqn_4}, we have $l=2x+k$ for $1\leq k\leq x$, and hence the inequality in Eqn~\eqref{eqn:eqn_2} reduces to
\begin{equation}\label{eqn:eqn_6}
(k-1)  < \sum_{i=1}^{2x+1-k} \dfrac{1}{n_{2x+k+i} - 2} < (k-1) + \dfrac{1}{2} \ \mbox{ for } 1\leq k\leq x.
\end{equation}
By  Eqn.~\eqref{eqn:eqn_2}, we know that there is an one to one correspondence between integer $(\geq 3)$ solutions of the inequality in Eqn.~\eqref{eqn:eqn_6} and complete $m$-partite graphs $G$ with $\lambda_G<0$.   Further, by
part $(i)$ of $Lemma~\ref{lem:lem_1}$, the inequality in Eqn.~\eqref{eqn:eqn_5} has infinitely many  integer ($\geq 3$) solutions   if and only if $(2x+1-k)-(k-1)\geq 1$ and $1\leq k\leq x.$ Hence there are infinitely many complete $m$-partite graph $G$ with $\lambda_G<0$  if and only if $1\leq k\leq \frac{2x+1}{2}.$ Since $m\geq 5$ and $m=4x+1$, so $\{k: 1\leq k\leq \frac{2x+1}{2}\}$ is a non-empty set and hence the result holds true for this case.

Let $m = 4x+2$. Then by Eqn.~\eqref{eqn:eqn_4}, we have $l=2x+k$ for $2 \le k \le x+1$, and hence the inequality in Eqn~\eqref{eqn:eqn_2} reduces to
\begin{equation}\label{eqn:eqn_7}
(k-2)+\dfrac{1}{2}  < \sum_{i=1}^{2x+2-k} \dfrac{1}{n_{2x+k+i} - 2} < (k-1) \ \mbox{ for } 2\leq k\leq x+1. 
\end{equation}
Thus, using  Eqn.~\eqref{eqn:eqn_2},  part $(ii)$ of $Lemma~\ref{lem:lem_1}$  and  arguments similar to preceding cases  yields that there are infinitely many complete $m$-partite graph $G$ with $\lambda_G<0$  if and only if $2\leq k\leq x+1.$ Since $m\geq 5$ and $m=4x+2$, so $\{k: 2\leq k\leq x+1\}$ is a non-empty set and hence the result holds true for this case.

Let $m = 4x+3$. Then by Eqn.~\eqref{eqn:eqn_4}, we have $l=2x+k$ for $2\leq k\leq x+2$, and hence the inequality in Eqn~\eqref{eqn:eqn_2} reduces to
\begin{equation}\label{eqn:eqn_8}
(k-2)  < \sum_{i=1}^{2x+3-k} \dfrac{1}{n_{2x+k+i} - 2} < (k-2) + \dfrac{1}{2} \ \mbox{ for } 2\leq k\leq x+2.
\end{equation}
Thus, using  Eqn.~\eqref{eqn:eqn_2},  part $(i)$ of $Lemma~\ref{lem:lem_1}$  and  arguments similar to preceding cases  yields that there are infinitely many complete $m$-partite graph $G$ with $\lambda_G<0$  if and only if  $2\leq k\leq x+2.$ Since $m\geq 5$ and $m=4x+3$, so $\{k: 2\leq k\leq x+2\}$ is a non-empty set and hence the result follows. This completes the proof.
\end{proof}

\begin{rem}\label{rem:lam_G}
From the above discussion, it  clear that  there are infinitely many complete multipartite graphs $G$ with $\cof D(G)$ satisfying $\lambda_G <0$. We also  know that similar assertion is  true for   $\lambda_G >0$ and 
as well as for $\lambda_G = 0$. Now we consider a special case of the complete tripartite graph.

Let $G$ be a complete tripartite graph. Then by Theorem~\ref{thm:beta-l-1}, we have $\det D(G) = 0$ if and only if $G = K_{2,2,n}$ for $n > 1$. By Theorem~\ref{thm:gamma-l-1}, we have $\cof D(G) = 0$ if and only if $G = K_{1,1,4}$ or $K_{2,2,n}$ for $n > 1$. The complete tripartite graph $ K_{1,1,n-2}$ appeared in the study of completely positive graphs, denoted by $T_n$ (for details see~\cite{Br}), and for notational convenience we are continuing with  $T_n$ instead of $ K_{1,1,n-2}$. 
\vspace*{-.1cm}
\begin{figure}[ht]
\centering
\begin{tikzpicture}[scale=0.7]
\Vertex[label=1,position=below,size=.1,color=black]{A} 
\Vertex[x=4,label=2,position=below,size=.1,color=black]{B} 
\Vertex[x=2,y=1,label=3,position=below,size=.1,color=black]{C} 
\Vertex[x=2,y=2,label=4,position=below,size=.11,color=black]{D}
\Edge[lw=1.5pt](A)(B)
\Edge[lw=1.5pt](A)(C)
\Edge[lw=1.5pt](A)(D)
\Edge[lw=1.5pt](B)(C)
\Edge[lw=1.5pt](B)(D)
\end{tikzpicture} \hspace{5mm} \begin{tikzpicture}[scale=0.7]
\Vertex[label=1,position=below,size=.1,color=black]{A} 
\Vertex[x=4,label=2,position=below,size=.1,color=black]{B} 
\Vertex[x=2,y=1,label=3,position=below,size=.1,color=black]{C} 
\Vertex[x=2,y=2,label=4,position=below,size=.11,color=black]{D}
\Vertex[x=2,y=3,label=5,position=below,size=.1,color=black]{E} 
\Edge[lw=1.5pt](A)(B)
\Edge[lw=1.5pt](A)(C)
\Edge[lw=1.5pt](A)(D)
\Edge[lw=1.5pt](A)(E)
\Edge[lw=1.5pt](B)(C)
\Edge[lw=1.5pt](B)(D)
\Edge[lw=1.5pt](B)(E)
\end{tikzpicture}  \hspace{5mm}  \begin{tikzpicture}[scale=0.6]

\Vertex[label=1,position=below,size=.1,color=black]{A} 
\Vertex[x=4,label=2,position=below,size=.1,color=black]{B} 
\Vertex[x=2,y=1,label=3,position=below,size=.1,color=black]{C} 
\Vertex[x=2,y=2,label=4,position=below,size=.11,color=black]{D}
\Vertex[x=2,y=3,label=5,position=below,size=.1,color=black]{E} 
\Vertex[x=2,y=4,label=6,position=below,size=.1,color=black]{F}
\Edge[lw=1.5pt](A)(B)
\Edge[lw=1.5pt](A)(C)
\Edge[lw=1.5pt](A)(D)
\Edge[lw=1.5pt](A)(E)
\Edge[lw=1.5pt](A)(F)
\Edge[lw=1.5pt](B)(C)
\Edge[lw=1.5pt](B)(D)
\Edge[lw=1.5pt](B)(E)
\Edge[lw=1.5pt](B)(F)
\end{tikzpicture}
\caption{$T_4,T_5,T_6$}\label{fig:M1}
\end{figure}
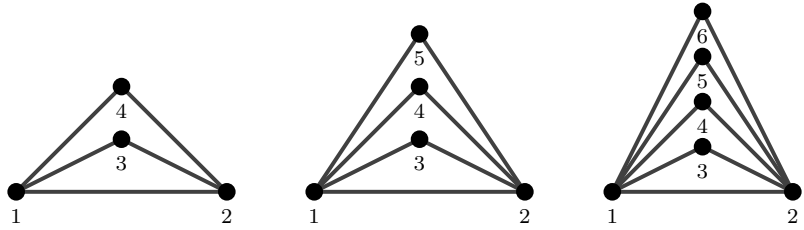
For $n\geq 3$, $T_n$ can be seen as a graph consisting of $(n-2)$ triangles with a common base (see Figure~\ref{fig:M1}). Note that  $\cof D(T_6)=0$ and $\lambda_{T_n} = -\frac{2}{n-6}$  $(n \neq 6)$ which implies that $\lambda_{T_n} < 0$ if and only if $G=T_n$ for $n\geq 7$.
\end{rem}

Let $G = (V,E)$ be a multi-block graph on $|V|$ vertices with blocks $G_t, \ 1\leq t \leq b$. If  $ \cof  D(G_t)\neq 0$ for all $1\leq t\leq b$, then define  
 \begin{equation}\label{eqn:lamda-G}
\lambda_G = \sum_{t=1}^{b} \lambda_{G_{t}}.
\end{equation}

The following result gives us the determinant of the distance matrix of multi-block graphs. We omit the proof as it follows from Theorem~\ref{Thm:cof-det}.

\begin{theorem}\label{Thm:det_multi}
Let $G=(V,E)$ be a multi-block graph with blocks $G_t$, $1\leq t\leq b.$  If  $ \cof  D(G_t)\neq 0$ for all $1\leq t\leq b$, then the determinant of  $D(G)$ is given by 
$$\det D(G)= \lambda_G \prod_{t=1}^b \cof D(G_t) = \lambda_G \ \cof D(G).$$
\end{theorem}

Given a multi-block graph $G$ under the assumption
of the above theorem (which implies $\cof D(G)\neq 0$ by Theorem~\ref{Thm:cof-det}), the $\det D(G) = 0$ if and only if $\lambda_G = 0$. By Remark~\ref{rem:lam_G} and Eqn.~\eqref{eqn:lamda-G}, we can choose  the blocks such that $\lambda_G = 0$ and $\cof D(G) \ne 0$. Two such possible rearrangements are given in the following examples.

\begin{ex}
Let $G$ be a multi-block graph with blocks $G_i=T_{n_i}$,  $(n_i\neq 6)$ for all $1\leq i\leq b$. Suppose $G$ contain $b_1$ blocks of $T_3$, $b_2$ blocks of $T_4$ and $b_3$ blocks of $T_5$.  Now for each $T_3$ we can associate $x$ blocks of $T_n$, where $n=3x+6$. Then $$\lambda_{T_3}+x\lambda_{T_n}=\frac{2}{3}-\frac{2x}{3x+6-6}=0.$$ Similarly, for each $T_4$ we can associate $y$ blocks of $T_n$, where $n=2y+6$ and for each $T_5$ we can associate $z$ blocks of $T_n$, where $n=z+6$. Then $\lambda_G =\sum_{i} \lambda_{G_i}=0$ and $D(G)$ is not invertible. Consequently, for different choices $x,y$ and $z$, we can generate infinitely many multi-block graphs $G$ (where $\cof D(G_i) \neq 0$ for all $i$ ) with $\det D(G)=0$.
\end{ex}

\begin{ex}
Let $G$ be a multi-block graph consisting of  a  $K_4$, $K_{m,m}$ ($m\neq 2$) and $b$ blocks of $T_n$, where $b=9m-4$, $n=6+8m$. Then 
$$\lambda_G = \lambda_{K_4}+\lambda_{K_{m,m}}+b \lambda_{T_n}= \frac{3}{4}+\frac{3m-2}{2m}-\frac{2b}{n-6}= \frac{9m-4}{4m}-\frac{2(9m-4)}{6+8m-6}=0,$$
and $D(G)$ is not invertible. Consequently, for different values of $m\in \mathbb{N}\setminus \{2\}$, it is possible to produce  infinite number of multi-block graphs $G$ (where $\cof D(G_i) \neq 0$ for all $i$) with $\det D(G)=0$.
\end{ex}

Let $G = (V,E)$  be  a  complete $m$-partite graph $K_{n_1,n_2,\cdots,n_m}$ with $\cof D(G)\neq 0$. Suppose $v$ is a vertex of $G$, then we define a $|V|$-dimensional column vector $\mu_G$  as follows:
\begin{equation}\label{eqn:mu_G-single}
\mu_{G}(v)= \displaystyle \dfrac{ 1}{\gamma_{n_1, n_2,\cdots,n_m}} \sum_{i=1}^m \sum_{v\in V_{n_i}} \displaystyle\prod_{ j\neq i} (n_j-2).
\end{equation}
It is easy to observe that  the definition of $\lambda_G$ and $\mu_G$ as  defined in Eqns.~\eqref{eqn:lamda-G} and~\eqref{eqn:mu_G-single} agrees with the same, in~\cite{Bp3,Hou3}. For $G=K_{n_1,n_2,\cdots,n_m}$, in~\cite[page $22$]{Zhou1} the $\mu_G$ is denoted as $\beta$ and  also shown $\beta^t \mathds{1}=1$ and $\beta^tD(G)  = \lambda_G \mathds{1}$,  that is, $\mu_G^t \mathds{1}=1 \mbox{ and } \mu_G^tD(G) = \lambda_G \mathds{1},$ where $\lambda_G$ is defined as in Eqn.~\eqref{eqn:lambda_G-single}.

Next,  we define the Laplacian-like matrix for complete $m$-partite graphs.  Given a complete $m$-partite graph  $ K_{n_1,n_2,\cdots,n_m}$ with $m$-partition of the vertex set  $V_{i}$   for $1 \leq i \leq m$, respectively, we define a matrix  $\mathcal{L}=\mathcal{L}(K_{n_1,n_2,\cdots,n_m})=[\mathcal{L}_{uv}]$ with $\mathcal{L}\mathds{1} = \mathbf{0}$ and $\mathds{1}^t \mathcal{L} = \mathbf{0}$, called  Laplacian-like matrix of $ K_{n_1,n_2,\cdots,n_m}$, where
\begin{equation}\label{eqn:Laplike-single}
\mathcal{L}_{uv}=
\begin{cases}
\dfrac{(n_i-1)\beta_{\widehat{n_i}}-2 \gamma_{\widehat{n_i}}}{2 \gamma_{n_1, n_2,\cdots,n_m}}=a_i & \text{if} \ u = v ,u \in V_{i} \mbox{ for  } 1 \leq i \leq m;\\
\\
-\dfrac{\beta_{\widehat{n_i}}}{2\gamma_{n_1, n_2,\cdots,n_m}} =b_i & \text{if} \ u \neq v ,u,v \in V_{i} \mbox{ for  } 1 \leq i \leq m;\\
\\
\dfrac{\prod_{l \neq i,j}(n_l - 2)}{\gamma_{n_1, n_2,\cdots,n_m}}=c_{ij}& \text{if} \ u \sim v ,u \in V_{i},v\in V_{j} \mbox{ for  } 1 \leq i,j \leq m.\\
\end{cases}
\end{equation}
It can be checked  for $m=2$, the above definition agrees with  the  Laplacian-like matrix $\mathcal{L}(K_{n_1,n_2})=[\mathcal{L}(K_{n_1,n_2})_{uv}]$ for $K_{n_1,n_2}$  as defined in~\cite[Eqn.(4)]{Hou3}.  The next lemma gives the relation between $\mathcal{L}$ and the distance matrix of $K_{n_1,n_2,\cdots,n_m}$.
\begin{lem}\label{lem:LD+I-single}
Let    $\mathcal{L}$ be the Laplacian-like matrix and $ D(G) $ be the distance matrix of a complete $m$-partite graph $G=K_{n_1,n_2,\cdots,n_m}$, then $ \mathcal{L}D(G)+I=\mu_G \mathds{1}^t$, where $\mu_G$  is  as defined in Eqn.~\eqref{eqn:mu_G-single}.
\end{lem}
\begin{proof}
Let $V_1,V_2,\cdots,V_m$ be the partition of the vertex set for the complete $m$-partite graph $K_{n_1,n_2,\cdots,n_m}$.  Now we consider the following cases:\\

\noindent $\underline{\textbf{Case 1:}} \ u,v \in V_i\ $; where $1\leq i\leq m $.
\begin{itemize}

\item For  $u = v$,
\begin{align*} 
(\mathcal{L} D+I)_{uv}  &=  1 + (\mathcal{L}D)_{uv} = 1 + 2(n_i -1)b_i + \sum_{s\neq i} n_s c_{is} \\ 
					    &=  1 - \dfrac{(n_i - 1) \left(\gamma_{\widehat{n_i}} +\prod_{\substack{s = 1, s \neq i}}^m (n_s - 2) \right)}{\gamma_{n_1, n_2,\cdots,n_m}} + \dfrac{\gamma_{\widehat{n_i}}}{\gamma_{n_1, n_2,\cdots,n_m}}\\
					    & = 1 - \dfrac{(n_i - 2) \gamma_{\widehat{n_i}}+ \gamma_{\widehat{n_i}} +(n_i - 1) \prod_{\substack{s = 1, s \neq i}}^m (n_s - 2) }{\gamma_{n_1, n_2,\cdots,n_m}} + \dfrac{\gamma_{\widehat{n_i}}}{\gamma_{n_1, n_2,\cdots,n_m}}\\
					    & = \dfrac{ \prod_{\substack{s= 1,  s \neq i}}^m (n_s - 2)}{\gamma_{n_1, n_2,\cdots,n_m}} = \mu(u).
\end{align*}
Similar calculations helps us to obtain the remaining cases.
\item  If $u\neq v$, then $\ds (\mathcal{L} D+I)_{uv} = (\mathcal{L}D)_{uv} = 2a_i + 2(n_i - 2)b_i + \sum_{s \neq i} n_s c_{is} = \mu_G(u).$
\end{itemize}
\vspace*{.2cm}
\noindent $\underline{\textbf{Case 2:}} \ u \in V_i \text{ and } v \in V_j\ $;  where $1\leq i,j\leq m $ with $\ i\neq j$.
$$(\mathcal{L} D+I)_{uv} = (\mathcal{L}D)_{uv} = a_i + (n_i - 1)b_i + 2(n_j -1)c_{ij} + \sum_{s \neq i,j} n_s c_{is} = \mu_G(u).$$
Combining all the above cases, we get $(\mathcal{L} D+I)_{uv} =\mu_G(u)$, for any $u,v$ and  the result follows.
\end{proof} 

Next we recall a definition from~\cite{Zhou1}, which is useful for our subsequent results. 
\begin{defn}
 Let $D$ and $\mathcal{L}$ be two $n \times n$ matrices,  $\beta$ be an $n \times 1$ column vector and $\lambda$ be a  number. Then
\begin{itemize}
\item[1.]  $D$ is a left LapExp($\lambda,  \beta, \mathcal{L}$) matrix,  if  
$\beta^T \mathds{1},\ \mathcal{L} \mathds{1}=0, \ \beta^T D=\lambda \mathds{1}^T \mbox{ and } \mathcal{L}D+I=\beta \mathds{1}^T$,
\item[2.] $D$ is a right LapExp($\lambda, \beta, \mathcal{L}$) matrix,  if $\beta \mathds{1}^T=1,\  \mathds{1}^T \mathcal{L}=0,\ D\beta=\lambda\mathds{1}   \mbox{ and } D\mathcal{L} +I = \mathds{1}\beta^T$.
\end{itemize}
\end{defn}

\begin{rem}\label{rem:lap-exp}
In view of Lemma~\ref{lem:LD+I-single} and~\cite{Hou3,Zhou1},  for $m\geq 2$, if $D(G)$ is the distance matrix of a complete m-partite graph $G$, then $D(G)$ is a left LapExp($\lambda_G,  \mu_G, \mathcal{L}$) matrix, where $\lambda_G,  \mu_G, \mathcal{L}$ are as defined in Eqns.~\eqref{eqn:lambda_G-single}, \eqref{eqn:mu_G-single} and \eqref{eqn:Laplike-single}, respectively. Similar calculation also leads to $D(G)$ is a right LapExp($\lambda_G,  \mu_G, \mathcal{L}$) matrix.  Further, using~\cite[Lemma 3.1]{Zhou1} if $\lambda \neq 0$, then  D is invertible and $D^{-1}= - \mathcal{L} + \dfrac{1}{\lambda} \mu_G \mu_G^T$. 
\end{rem}

Now,  we  extend the definition of $\mu_G$  and the Laplacian-like matrix $\mathcal{L}$, for  multi-block graphs. Let $G=(V,E)$ be a multi-block graph on $|V|$ vertices with blocks $G_t=(V_t,E_t)$, $1\leq t\leq b$ such that $ \cof  D(G_t)\neq 0$ for $1\leq t\leq b$. Let $\mu_G$ (we will use $\mu$ if there is no scope of confusion) be a $|V|$-dimensional column vector defined as follows. If a vertex $v$ belongs to $k$  blocks of $G$, then
\begin{equation}\label{eqn:mu(v)}
\mu_G(v) = \sum_{t=1}^{b} \mu_{G_t}(v)- (k-1).
\end{equation}
Next  each block $G_t$ is also considered  as a graph on vertex set $V$ with perhaps isolated vertices, and  let its edge set be $E_t$ (\emph{i.e.,} $G_t$ is a  graph on $|V_t|$ vertices, consider it as a graph on vertex  set $V$). Let $\mathcal{L}$  be the $|V|\times|V|$ matrix defined as above for the vertices of $G_t$ and $0$ for others. Define 
\begin{equation}\label{eqn:lap-multi-block}
\mathcal{L} = \sum_{t=1}^{b}\mathcal{L}(G_t).
\end{equation}
It can be seen that $\mathcal{L}_{uv} = 0$ if $u$ and $v$ are not in the same block and $\mathcal{L}\mathds{1} = \mathds{1}^t \mathcal{L} = \mathbf{0}$.

In view of Remark~\ref{rem:lap-exp} and \cite[Lemma 7]{Hou3}, the result below  for a multi-block graph follows from~\cite[Lemmas 4.5, 4.7 and 4.8]{Zhou1}.
\begin{lem}\label{lem:LD+I}
Let $\mathcal{L}$ be the Laplacian-like matrix and $ D(G) $ be the distance matrix of a multi-block graph $ G $. Then,  $D(G)$ is a left LapExp($\lambda_G,  \mu_G, \mathcal{L}$) matrix, {\it i.e.},
 $$\mu_G^t \mathds{1}=1, \mathcal{L}\mathds{1}=0,  D(G) \mu = \lambda_G \mathds{1},  \mbox{ and } \mathcal{L}D(G)+I=\mu_G \mathds{1}^t,$$
 where $\lambda_G,  \mu_G, \mathcal{L}$ are as defined in Eqns.~\eqref{eqn:lamda-G}, \eqref{eqn:mu(v)} and \eqref{eqn:lap-multi-block}, respectively.
\end{lem}
It  can be seen that,  $D(G)$ is also a right LapExp($\lambda_G,  \mu_G, \mathcal{L}$) matrix. Further, by Theorem~\ref{Thm:det_multi},   $ D(G)$ is   non-singular implies that $\lambda_G\neq 0$ and hence  by Lemma~\ref{lem:LD+I} and~\cite[Lemma 3.1]{Zhou1} we have the following theorem.
\begin{theorem}\label{Thm:D-inv-multi}
Let $G$ be  multi-block graph with blocks $G_t$, $1\leq t\leq b.$ Let $\mathcal{L}$ be the Laplacian-like matrix and $ D(G) $ be the distance matrix of  $ G $. If  $\det D(G) \neq 0$ and $\cof  D(G_i)\neq 0 $,  $1\leq i\leq b$, then  
$$D(G)^{-1} = - \mathcal{L} + \dfrac{1}{\lambda_G}\mu \mu^t,$$
where $\lambda_G$ and $\mu$ are as defined in Eqns.~\eqref{eqn:lamda-G} and \eqref{eqn:mu(v)},  respectively.
\end{theorem}

\begin{rem}\label{rem:cf=0}
Let $D$ be an $n\times n$ invertible matrix with  $\cof D=0$. Now we show that  $D^{-1}$ cannot be written as a rank one perturbation of a Laplacian-like matrix. Suppose on the contrary
\begin{equation}\label{eqn:ddd}
D^{-1}=\mathcal{L}+\mu \mu^t,
\end{equation}
where $\mathcal{L}\mathds{1}=\mathds{1} \mathcal{L}= \mathbf{0}$ and $\mu=(\mu_1,  \mu_2, \cdots,\mu_n)^t$. Let $\eta= D^{-1}\mathds{1}$. Then $$\mathds{1}^t \eta
=\mathds{1}^t D^{-1}\mathds{1}=\mathds{1}^t \left(\dfrac{1}{\det D} \textup{ Adj } D\right)\; \mathds{1} =  \dfrac{1}{\det D} \mathds{1}^t \textup{ Adj } D \ \mathds{1} = \dfrac{\cof D}{\det D}=0.$$ 
Next, using $\mathcal{L}\mathds{1}=\mathbf{0}$, Eqn~\eqref{eqn:ddd} reduces to $D^{-1} \mathds{1} = \mu \mu^t \mathds{1}$ and hence $$\eta=D^{-1} \mathds{1}=\left(\mu_1 \sum_{i=1}^n \mu_i, \mu_2 \sum_{i=1}^n \mu_i,\cdots,\mu_n \sum_{i=1}^n \mu_i\right)^t.$$
Thus,  $0=\mathds{1}^t \eta=\ds\left(\sum_{i=1}^n \mu_i\right)^2$ and hence $\mathds{1}^t \mu= \mu^t \mathds{1}=0.$ Therefore, $D^{-1} \mathds{1} = \mu \mu^t \mathds{1}=\mathbf{0}$, which is a contradiction.
\end{rem}

Let $G$ be a multi-block graph with one of its block is of zero cofactor. Then, by Theorem~\ref{Thm:cof-det},  we have $\cof D(G)=0$, whereas  $\det D(G)$ may or may not be zero. If $\det D(G)\neq 0$, then by Remark~\ref{rem:cf=0}, the inverse of $D(G)$ cannot be written as a rank one perturbation of a Laplacian-like matrix. 

Therefore, the strategy adopted in this section is not applicable for the graphs with one of its block is of zero cofactor. The next section deals with one such class of multi-block graphs containing exactly one block as  $T_6$.

\section{Determinant and Inverse of $D(T_6 \circledcirc T_n^{(b)})$ }\label{sec:T6-Tn-b}

 Let $T_6 \circledcirc T_n^{(b)}$ denote the multi-block graph with exactly one block as $T_6$ and $b \ (\geq 1)$ blocks of  $T_n$, for a fixed $n$, where $n\neq 6$ and with a central cut vertex, which is not a base vertex (see Figure~\ref{fig:33}). Let $D(T_6 \circledcirc T_n^{(b)})$ be the distance matrix of $T_6 \circledcirc T_n^{(b)}$. Using Theorem~\ref{Thm:det_multi}, the  $\det D(T_6 \circledcirc T_n^{(b)})\neq 0$ and hence the distance matrix is invertible. For sake of completeness we state the result for determinant without proof.

\begin{theorem}
Let $D(T_6 \circledcirc T_n^{(b)})$ be the distance matrix of the graph $T_6 \circledcirc T_n^{(b)}$. Then the determinant of $D(T_6 \circledcirc T_n^{(b)})$ is given by
$$\det D(T_6 \circledcirc T_n^{(b)}) = (-1)^{nb+1} 2^{(n-3)b+4}(n-6)^b.$$
\end{theorem}

\begin{figure}[ht]
 \begin{subfigure}{0.45\textwidth}
  \centering
  
    \begin{tikzpicture}[scale=0.45]
\Vertex[position=right,size=.01,color=red]{A} 
\Vertex[y=1,position=below,size=.01,color=black]{B} 
\Vertex[y=2,position=below,size=.01,color=black]{C}
\Vertex[y=3,position=below,size=.01,color=black]{D} 
\Vertex[y=4,position=below,size=.01,color=black]{E}
\Vertex[x=3,y=5,position=right,size=.01,color=black]{F} 
\Vertex[x=-3,y=5,position=left,size=.01,color=black]{G}

\Vertex[y=-1,position=below,size=.01,color=black]{B1} 
\Vertex[y=-2,position=below,size=.01,color=black]{C1}
\Vertex[y=-3,position=below,size=.01,color=black]{D1} 
\Vertex[y=-4,position=below,size=.01,color=black]{E1}
\Vertex[x=3,y=-5,position=below,size=.01,color=black]{F1} 
\Vertex[x=-3,y=-5,size=.01,position=below,color=black]{G1}

\Vertex[x=1,size=.01,position=right,color=black]{B2} 
\Vertex[x=2,size=.01,position=right,color=black]{C2}
\Vertex[x=3,size=.01,position=right,color=black]{D2}
\Vertex[x=4,y=-3,size=.01,position=right,color=black]{F2} 
\Vertex[x=4,y=3,size=.01,position=right,color=black]{G2}

\Vertex[x=-1,size=.01,position=left,color=black]{B3} 
\Vertex[x=-2,size=.01,position=left,color=black]{C3}
\Vertex[x=-3,size=.01,position=left,color=black]{D3}
\Vertex[x=-4,size=.01,position=left,color=black]{H3} 
\Vertex[x=-5,y=3,size=.01,position=left,color=black]{G3}
\Vertex[x=-5,y=-3,position=left,size=.01,color=black]{F3}

\Edge[lw=1.5pt](A)(F)
\Edge[lw=1.5pt](A)(G)
\Edge[lw=1.5pt](B)(F)
\Edge[lw=1.5pt](B)(G)
\Edge[lw=1.5pt](C)(F)
\Edge[lw=1.5pt](C)(G)
\Edge[lw=1.5pt](D)(F)
\Edge[lw=1.5pt](D)(G)
\Edge[lw=1.5pt](E)(F)
\Edge[lw=1.5pt](E)(G)
\Edge[lw=1.5pt](F)(G)

\Edge[lw=1.5pt](A)(F1)
\Edge[lw=1.5pt](A)(G1)
\Edge[lw=1.5pt](B1)(F1)
\Edge[lw=1.5pt](B1)(G1)
\Edge[lw=1.5pt](C1)(F1)
\Edge[lw=1.5pt](C1)(G1)
\Edge[lw=1.5pt](D1)(F1)
\Edge[lw=1.5pt](D1)(G1)
\Edge[lw=1.5pt](E1)(F1)
\Edge[lw=1.5pt](E1)(G1)
\Edge[lw=1.5pt](F1)(G1)

\Edge[lw=1.5pt](A)(F3)
\Edge[lw=1.5pt](A)(G3)
\Edge[lw=1.5pt](B3)(F3)
\Edge[lw=1.5pt](B3)(G3)
\Edge[lw=1.5pt](C3)(F3)
\Edge[lw=1.5pt](C3)(G3)
\Edge[lw=1.5pt](D3)(F3)
\Edge[lw=1.5pt](D3)(G3)
\Edge[lw=1.5pt](H3)(F3)
\Edge[lw=1.5pt](H3)(G3)
\Edge[lw=1.5pt](F3)(G3)

\Edge[lw=1.5pt](A)(F2)
\Edge[lw=1.5pt](A)(G2)
\Edge[lw=1.5pt](B2)(F2)
\Edge[lw=1.5pt](B2)(G2)
\Edge[lw=1.5pt](C2)(F2)
\Edge[lw=1.5pt](C2)(G2)
\Edge[lw=1.5pt](D2)(F2)
\Edge[lw=1.5pt](D2)(G2)
\Edge[lw=1.5pt](F2)(G2)

\end{tikzpicture}
    \caption{$T_6  \circledcirc T_7^{(3)}$}
    \label{fig:1}
  \end{subfigure} \hspace{15mm}
   \begin{subfigure}{0.4\textwidth}
  \centering
    \begin{tikzpicture}[scale=0.7]
\Vertex[position=right,size=.01,color=red]{A} 
\Vertex[x=1.5,y=2.5,position=right,size=.01,color=black]{F} 
\Vertex[x=-1.5,y=2.5,position=left,size=.01,color=black]{G}

\Vertex[x=1.5,y=-2.5,position=below,size=.01,color=black]{F1} 
\Vertex[x=-1.5,y=-2.5,size=.01,position=below,color=black]{G1}

\Vertex[x=1,size=.01,position=right,color=black]{B2} 
\Vertex[x=2,size=.01,position=right,color=black]{C2}
\Vertex[x=3,size=.01,position=right,color=black]{D2}
\Vertex[x=4,y=-2,size=.01,position=right,color=black]{F2} 
\Vertex[x=4,y=2,size=.01,position=right,color=black]{G2}

\Vertex[x=-3,y=1.5,size=.01,position=left,color=black]{G3}
\Vertex[x=-3,y=-1.5,position=left,size=.01,color=black]{F3}

\Edge[lw=1.5pt](A)(F)
\Edge[lw=1.5pt](A)(G)
\Edge[lw=1.5pt](F)(G)

\Edge[lw=1.5pt](A)(F1)
\Edge[lw=1.5pt](A)(G1)
\Edge[lw=1.5pt](F1)(G1)

\Edge[lw=1.5pt](A)(F3)
\Edge[lw=1.5pt](A)(G3)
\Edge[lw=1.5pt](F3)(G3)

\Edge[lw=1.5pt](A)(F2)
\Edge[lw=1.5pt](A)(G2)
\Edge[lw=1.5pt](B2)(F2)
\Edge[lw=1.5pt](B2)(G2)
\Edge[lw=1.5pt](C2)(F2)
\Edge[lw=1.5pt](C2)(G2)
\Edge[lw=1.5pt](D2)(F2)
\Edge[lw=1.5pt](D2)(G2)
\Edge[lw=1.5pt](F2)(G2)

\end{tikzpicture}
\vspace*{.7cm}
    \caption{$T_6 \circledcirc T_3^{(3)}$}
    \label{fig:2}
  \end{subfigure}
  \caption{}\label{fig:33}
\end{figure}

We have found that the inverse of $D(T_6 \circledcirc T_n^{(b)})$ is of a similar form as in~\cite[Section 3]{JD}. To be specific, we found a matrix $R$ such that  $D(T_6 \circledcirc T_n^{(b)})^{-1}$  is a linear combination  of Laplacian matrix $L(T_6 \circledcirc T_n^{(b)})$, rank one matrix $J$ and the matrix $R$. For notational convenience,  we write  $D$ and $L$ to denote $D(T_6 \circledcirc T_n^{(b)})$ and $L(T_6 \circledcirc T_n^{(b)})$,  respectively. With suitable choice of vertex indexing, the block form of $D$ and $L$ can be written as:
\begin{small}
\begin{equation}\label{eqn:D}
D =
\left[
\begin{array}{c|c|c|c|c|c|c}
D_1 & D_2 &D_2  &\cdots &D_2 &D_2 &\mathbf{d}_3\\
\hline
D_2^t & D_4 & D_5  &\cdots &D_5 &D_5 & \mathbf{d}_6 \\
\hline
D_2^t &D_5 & D_4  &\cdots &D_5 &D_5 & \mathbf{d}_6 \\
\hline
\vdots & \vdots &   \vdots &\ddots & \vdots & \vdots& \vdots\\
\hline
D_2^t & D_5 & D_5  &\cdots &D_4 & D_5& \mathbf{d}_6 \\
\hline
D_2^t & D_5 & D_5  &\cdots & D_5 &D_4 & \mathbf{d}_6 \\
\hline
\mathbf{d}_3^t &\mathbf{d}_6^t   & \mathbf{d}_6^t &\cdots &\mathbf{d}_6^t & \mathbf{d}_6^t & 0 \\
\end{array}
\right], \ 
L=
\left[
\begin{array}{c|c|c|c|c|c|c}
L_1 & \mathbf{0}       & \mathbf{0} &\cdots &\mathbf{0} &\mathbf{0} &\textbf{l}_1 \\
\hline
\mathbf{0} &L_2        & \mathbf{0} &\cdots &\mathbf{0} &\mathbf{0} &\textbf{l}_2 \\
\hline
\mathbf{0} &\mathbf{0} & L_2        &\cdots &\mathbf{0} &\mathbf{0} &\textbf{l}_2\\
\hline
\vdots     & \vdots     & \vdots     &\ddots & \vdots    & \vdots & \vdots \\
\hline
\mathbf{0} &\mathbf{0} & \mathbf{0} &\cdots &L_2        &\mathbf{0} & \textbf{l}_2\\
\hline
\mathbf{0} &\mathbf{0} & \mathbf{0} &\cdots    &\mathbf{0}  &L_2 & \textbf{l}_2\\
\hline
\textbf{l}_1^t &\textbf{l}_2^t &\textbf{l}_2^t  &\cdots &\textbf{l}_2^t   &\textbf{l}_2^t & 2(b+1) \\
\end{array}
\right], 
\end{equation}
where 
$$
D_1 = \left[
\begin{array}{c|c}
J_2 - I_2& J_{2 \times 3}\\
\hline
J_{3 \times 2} & 2(J_3-I_3)
\end{array}
\right], \ 
D_2 =  \left[
\begin{array}{c|c}
2J_2 & 3J_{2 \times (n-3)}\\
\hline
3J_{3 \times 2} & 4J_{3 \times (n-3)}\\
\end{array}
\right], \ 
\mathbf{d}_3= \left[
\begin{array}{c}
\mathds{1}_2 \\
\hline
2\mathds{1}_3
\end{array}
\right],
$$

$$
D_4 = \left[
\begin{array}{c|c}
J_2 - I_2& J_{2 \times (n-3)}\\
\hline
J_{(n-3) \times 2} & 2(J_{n-3}-I_{n-3})
\end{array}
\right],
D_5 = \left[
\begin{array}{c|c}
2J_2 & 3J_{2 \times (n-3)}\\
\hline
3J_{(n-3) \times 2} & 4J_{n-3}\\
\end{array}
\right], \ 
\mathbf{d}_6= \left[
\begin{array}{c}
\mathds{1}_2 \\
\hline
2\mathds{1}_{n-3}
\end{array}
\right],$$

$$
L_1 = \left[
\begin{array}{c|c}
6I_2- J_2             &-J_{2 \times 3}\\
\hline
-J_{3 \times 2}                 &2I_3\\
\end{array}
\right], \ 
L_2 = \left[
\begin{array}{c|c}
nI_2- J_2             &-J_{2 \times (n-3)}\\
\hline
-J_{(n-3) \times 2}                 &2I_{n-3}\\
\end{array}
\right],\ 
\textbf{l}_1 = \left[
\begin{array}{c}
-\mathds{1}_2\\
\hline
\mathbf{0}_{ 3 \times 1}\\
\end{array}
\right]
  \mbox{ and }
\textbf{l}_2 = \left[
\begin{array}{c}
-\mathds{1}_2\\
\hline
\mathbf{0}_{(n-3)\times 1}\\
\end{array}
\right].
$$  
\end{small}
\begin{rem}
Note that for $n=3$, some of the above  block matrices contains submatrices of the form $A_{p\times 0}$ or $A_{0\times p}$ and in these cases we will consider such matrices do not exist. For example with $n=3$, $D_2 =  \left[
\begin{array}{c}
2J_2 \\
\hline
3J_{3 \times 2}
\\
\end{array}
\right],\ D_4 = 
J_2 - I_2,\ 
D_5 = 2J_2,\ \mathbf{d}_6= \mathds{1}_2 $ etc. and we follow similar convention for rest of the section. 
\end{rem}

Next, we define the matrix $R$, whose block matrix is given by 

\begin{small}
\begin{equation}\label{eqn:R_T_6-T_nb}
R=
\left[
\begin{array}{c|c|c|c|c|c|c}
R_1 & R_2       & R_2 &\cdots & R_2 & R_2 &\textbf{r}_1 \\
\hline
R_2^t  & R_3    & R_4 &\cdots &R_4  &R_4 &\textbf{r}_2 \\
\hline
R_2^t &R_4& R_3       &\cdots &R_4  &R_4 &\textbf{r}_2 \\
\hline
\vdots     & \vdots     & \vdots     &\ddots & \vdots    & \vdots& \vdots \\
\hline
R_2^t  &R_4 & R_4 &\cdots &R_3   & R_4    & \textbf{r}_2 \\
\hline
R_2^t  &R_4 & R_4 &\cdots  & R_4  &R_3       & \textbf{r}_2 \\
\hline
\textbf{r}_1^t &\textbf{r}_2^t &\textbf{r}_2^t  &\cdots &\textbf{r}_2^t &\textbf{r}_2^t & r_3 \\
\end{array}
\right],
\end{equation}
where
$$
R_1 = \left[
\begin{array}{c|c}
\left(4b(b+1) - (3b+4)(n-6)\right) J_2 + 4(n-6)(b+1)I_2 &-\left(2b(b+1)+(n-6) \right) J_{2 \times 3}\\
\hline
-\left(2b(b+1)+(n-6) \right)J_{3 \times 2}  & \left(b(b+1)-(n-6) \right)J_3+(b+1)(n-6)I_3\\
\end{array}
\right],\ 
$$
$$
R_2 = \left[
\begin{array}{c|c}
-\left(2(b+1)(n-4) +(n-6)\right) J_2 &\left(2(b+1)-(n-6) \right) J_{2 \times (n-3)}\\
\hline
-\left((b+1)(n-4)-(n-6) \right)J_{3 \times 2}  & -\left((b+1)+(n-6) \right)J_{3 \times (n-3)}\\
\end{array}
\right],
$$
\noindent
$$
R_3 = (b+1)\left[
\begin{array}{c|c}
\left(n-2 -\dfrac{n-6}{b+1} \right) J_2 + (n-2)(n-6)I_2 &-\left(n-4+\dfrac{n-6}{b+1} \right) J_{2 \times (n-3)}\\
\hline
-\left(n-4+\dfrac{n-6}{b+1} \right)J_{(n-3) \times 2}  & \left(1 -\dfrac{n-6}{b+1} \right)J_{(n-3)}+(n-6)I_{(n-3)}\\
\end{array}
\right],\ 
$$
$$
R_4 = -(n-6)J_{(n-1)}, \ \textbf{r}_1 =  (n-6)\left[
\begin{array}{c}
\left(2b(b+1)-1 \right) \mathds{1}_2\\
\hline
-\left(b(b+1)+1 \right) \mathds{1}_3\\
\end{array}
\right],\ 
\textbf{r}_2 = -(n-6) \mathds{1}_{(n-1)}\  \text{and}\  r_3= -b^2(n-6).
$$
\end{small}

Now we state and prove the result which gives the inverse of $D$, whenever $n \neq 6$ and $b\geq 1$. 

\begin{theorem}\label{thm:D_inv_t6_tn}
Let $b \geq 1$  and $n\neq 6$. Let $D$ be the distance matrix of $T_6 \circledcirc T_n^{(b)}$. Then the inverse of $D$ is given by
$$
D^{-1} = -\dfrac{1}{2}L +\dfrac{1}{2(b+1)}J + \dfrac{1}{2(b+1)(n-6)} R,
$$ where $J$ is the matrix of all ones of conformal order, $L$ and $R$ are the matrices as defined in Eqns.~(\ref{eqn:D}) and (\ref{eqn:R_T_6-T_nb}), respectively.
\end{theorem}
\begin{proof}
Let us denote 
$$C = -\dfrac{1}{2}L +\dfrac{1}{2(b+1)}J + \dfrac{1}{2(b+1)(n-6)} R,$$ whose block form is given by
\begin{small}
\begin{equation}\label{eqn:D_inv_2}
C=\frac{1}{2(n-6)}
\left[
\begin{array}{c|c|c|c|c|c|c}
C_1 & C_2 &C_2  &\cdots &C_2 &C_2 &\mathbf{c}_3\\
\hline
C_2^t & C_4 & \mathbf{0}  &\cdots &\mathbf{0} &\mathbf{0} & \mathbf{c}_6 \\
\hline
C_2^t &\mathbf{0} & C_4  &\cdots &\mathbf{0} &\mathbf{0} & \mathbf{c}_6 \\
\hline
\vdots & \vdots & \vdots  & \ddots & \vdots & \vdots  & \vdots \\
\hline
C_2^t & \mathbf{0} & \mathbf{0}  &\cdots &C_4 & \mathbf{0} & \mathbf{c}_6 \\
\hline
C_2^t & \mathbf{0} & \mathbf{0}  &\cdots & \mathbf{0} &C_4 & \mathbf{c}_6 \\
\hline
\mathbf{c}_3^t &\mathbf{c}_6^t & \mathbf{c}_6^t  &\cdots &\mathbf{c}_6^t &\mathbf{c}_6^t & c_7 \\
\end{array}
\right],
\end{equation}
where
$$
C_1 = \left[
\begin{array}{c|c}
2\left(2b -(n-6)\right) J_2 - 2(n-6)I_2& -\left(2b -(n-6) \right) J_{2 \times 3}\\
\hline
-\left(2b -(n-6) \right) J_{3 \times 2} & bJ_3- (n-6)I_3
\end{array}
\right],\ 
C_2 = \left[
\begin{array}{c|c}
-2(n-4)J_2 & 2 J_{2 \times (n-3)}\\
\hline
(n-4) J_{3 \times 2} & -J_{3 \times (n-3)}
\end{array}
\right],\ 
$$

$$
\mathbf{c}_3= (n-6)\left[
\begin{array}{c}
(2b+1)\mathds{1}_2 \\
\hline
-b \mathds{1}_3
\end{array}
\right],\
C_4 = \left[
\begin{array}{c|c}
2(n-4)J_2 - 2(n-6)I_2& -2 J_{2 \times (n-3)}\\
\hline
-2 J_{(n-3) \times 2} & J_{n-3}- (n-6)I_{n-3}
\end{array}
\right],
$$
$$
\mathbf{c}_6= \left[
\begin{array}{c}
(n-6)\mathds{1}_2 \\
\hline
\mathbf{0}_{(n-3)\times 1}
\end{array}
\right] \ \text{and} \ 
c_7 = -(n-6)(3b+1).
$$
\end{small}

Consider the block matrix $Y = DC = (Y_{ij})$ of dimension $b+2$, where $Y_{ij}$ are block matrices of conformal order given by
\begin{small}
\begin{equation}\label{eqn:Y_1}
Y_{ij}=
\begin{cases}
D_1C_1+bD_2C_2^t+\mathbf{d}_3\mathbf{c}_3^t & \text{if} \ i = j = 1; \\
D_1C_2+D_2C_4+\mathbf{d}_3\mathbf{c}_6^t & \text{if} \ i =1, j =2,3,...,b+1; \\
D_1\mathbf{c}_3+ bD_2\mathbf{c}_6+c_7\mathbf{d}_3 & \text{if} \ i =1, j =b+2; \\
D_2^tC_1+D_4C_2^t+(b-1)D_5C_2^t+\mathbf{d}_6\mathbf{c}_3^t & \text{if} \ i =2,3,...,b+1, j =1; \\
D_2^tC_2+D_4C_4+\mathbf{d}_6\mathbf{c}_6^t & \text{if} \ i = j,\  i,j = 2,3,...,b+1;\\
D_2^tC_2+D_5C_4+\mathbf{d}_6\mathbf{c}_6^t & \text{if} \ i \neq j = 2,3,...,b+1;\\
D_2^t\mathbf{c}_3+D_4\mathbf{c}_6+(b-1)D_5\mathbf{c}_6+c_7\mathbf{d}_3 & \text{if} \ i =2,3,...,b+1, j = b+2;\\
\mathbf{d}_3^tC_1+b \mathbf{d}_6^tC_2^t & \text{if} \ i =b+2 ,j = 1;\\
\mathbf{d}_3^tC_2+\mathbf{d}_6^tC_1 & \text{if} \ i =b+2 ,j = 2,3,...,b+1;\\
\mathbf{d}_3^t\mathbf{c}_3+b \mathbf{d}_6^t\mathbf{c}_6 & \text{if} \ i = j =b+2.\\
\end{cases}
\end{equation}
\end{small}

We will show $Y=I$ to complete the proof. For the sake of simplicity, we compute $Y_{ij}$ in different steps.\\

\noindent {\textbf{Step 1 : }}$\underline{Y_{ij}, \mbox{ for } i = j = 1}$

\noindent Note that,\\
\begin{small} $ D_1C_1 = \left[
                                    \begin{array}{c|c}
                                       - \left(2b+(n-6) \right) J_2 + 2(n-6)I_2     & -b J_{2 \times 3}\\
                                       \hline
                                       -2(n-6)J_{3 \times 2} & 2(n-6)I_3
                                    \end{array}
                                 \right],\ \mathbf{d}_3\mathbf{c}_3^t  = (n-6) \left[
                                                  \begin{array}{c|c}
                                                     (2b+1)J_2 & -bJ_{2 \times 3}\\
                                                     \hline
                                                     2(2b+1)J_{3 \times 2} & -2bJ_3
                                                  \end{array}
                                                \right],                                \\
 \mbox{and }  D_2C_2^t  = \left[
                                 \begin{array}{c|c}
                                        -2(n-7)J_2                & (n-7)J_{2 \times 3}\\
                                        \hline
                                       -4(n-6)J_{3 \times 2} & 2(n-6)J_3
                                    \end{array}
                                 \right].$ \end{small}
Thus, for $ i=j=1$,  we have
$$Y_{ij} = \frac{1}{2(n-6)} \left[ D_1C_1 +bD_2C_2^t+\mathbf{d}_3\mathbf{c}_3^t \right] =I_{5}.$$

\noindent {\textbf{Step 2 : }} $\underline{Y_{ij},  \mbox{ for } i =1 ,\  j = 2,3,...,b+1}$\\

\begin{small}
\noindent Now, $ D_1C_2 = \left[
\begin{array}{c|c}
(n-4)J_2 & -J_{2 \times (n-3)}\\
\hline
\textbf{0}_{3\times 2} & \textbf{0}_{3 \times (n-3)}
\end{array}
\right], 
D_2C_4 = \left[
\begin{array}{c|c}
-2(n-5)J_2 & J_{2 \times (n-3)}\\
\hline
-2(n-6)J_{3\times 2} & \textbf{0}_{3 \times (n-3)}
\end{array}
\right] \mbox{ and }  \\
\mathbf{d}_3\mathbf{c}_6^t  = (n-6) \left[
\begin{array}{c|c}
J_2 & \textbf{0}_{2 \times (n-3)}\\
\hline
2J_{3\times 2} & \textbf{0}_{3 \times (n-3)}
\end{array}
\right].
$\end{small} Thus, for $ i =1 ,\  j = 2,3,...,b+1$,  we have
$$Y_{ij} = \frac{1}{2(n-6)} \left[ D_1C_2 +D_2C_4 +\mathbf{d}_3\mathbf{c}_6^t \right] = \textbf{0}_{5 \times (n-1)}.$$

\noindent {\textbf{Step 3 : }} $\underline{Y_{ij},  \mbox{ for } i =1 ,\   j = b+2}$\\
Now,
\begin{small}\begin{fleqn}
\begin{flalign*}
 D_1\mathbf{c}_3 &= (n-6)\left[
\begin{array}{c}
-(b-1) \mathds{1}_2 \\
\hline
2\mathds{1}_3 
\end{array}
\right], 
D_2\mathbf{c}_6 = 2(n-6)\left[
\begin{array}{c}
2 \mathds{1}_2 \\
\hline
3\mathds{1}_3 
\end{array}
\right] \mbox{ and } 
c_7\mathbf{d}_3  = -(n-6)(3b+1)\left[
\begin{array}{c}
\mathds{1}_2 \\
\hline
2\mathds{1}_3 
\end{array}
\right].
\end{flalign*}
\end{fleqn}\end{small}
Thus, for $ i =1 ,\   j = b+2$,  we have
$$Y_{ij} = \frac{1}{2(n-6)} \left[ D_1\mathbf{c}_3 +bD_2\mathbf{c}_6 +c_7\mathbf{d}_3  \right] = \textbf{0}_{5 \times 1}.$$

\noindent {\textbf{Step 4 : }} $\underline{Y_{ij},  \mbox{ for } i =2,3,...,b+1, j = 1}$\\

\begin{small}
$
 \mbox{ Next, } D_2^tC_1 = \left[
\begin{array}{c|c}
-(3(n-6)+2b)J_2 & (b+(n-6))J_{2 \times 3}\\
\hline
-6(n-6)J_{(n-3)\times 2} & 2(n-6)J_{(n-3) \times 3}
\end{array}
\right], 
D_4C_2^t = \left[
\begin{array}{c|c}
2 J_2 & -J_{2 \times 3}\\
\hline
\textbf{0}_{(n-3)\times 2} & \textbf{0}_{(n-3) \times 3}
\end{array}
\right], \\
D_5C_2^t = \left[
\begin{array}{c|c}
-2(n-7) J_2 & (n-7)J_{2 \times 3}\\
\hline
-4(n-6)J_{(n-3)\times 2} & 2(n-6)J_{(n-3) \times 3}
\end{array}
\right] \mbox{ and } \\
\mathbf{d}_6\mathbf{c}_3^t  = (n-6) \left[
\begin{array}{c|c}
(2b+1) J_2 & -bJ_{2 \times 3}\\
\hline
2(2b+1)J_{(n-3)\times 2} & -2bJ_{(n-3) \times 3}
\end{array}
\right].
$\end{small}
Thus, for $ i =2,3,...,b+1, j = 1$,  we have
$$Y_{ij} = \frac{1}{2(n-6)} \left[ D_2^tC_1 +D_4C_2^t +(b-1)D_5C_2^t +\mathbf{d}_6\mathbf{c}_3^t  \right] = \textbf{0}_{(n-1) \times 5}.$$

\noindent {\textbf{Step 5 : }} $\underline{Y_{ij},  \mbox{ for } i =j ,\  i,j = 2,3,...,b+1}$
\begin{fleqn}
\begin{flalign*}
 \mbox{Next, } D_2^tC_2 &= \left[
\begin{array}{c|c}
(n-4)J_2 & -J_{2 \times (n-3)}\\
\hline
\textbf{0}_{(n-3)\times 2} & \textbf{0}_{(n-3)}
\end{array}
\right], \ 
\mathbf{d}_6\mathbf{c}_6^t  = (n-6) \left[
\begin{array}{c|c}
J_2 & \textbf{0}_{2 \times (n-3)}\\
\hline
2J_{(n-3)\times 2} & \textbf{0}_{(n-3)}
\end{array}
\right] \mbox{ and }\\
D_4C_4 &= \left[
\begin{array}{c|c}
\left(4 -2(n-3) \right) J_2 + 2(n-6)I_2 & J_{2 \times (n-3)}\\
\hline
-2(n-6)J_{(n-3) \times 2} & 2(n-6)I_{(n-3)}
\end{array}
\right].
\end{flalign*}
\end{fleqn}

Thus, for $ i = j = 2,3,...,b+1$,  we have
$$Y_{ij} = \frac{1}{2(n-6)} \left[ D_2^tC_2 +D_4C_4 +\mathbf{d}_6\mathbf{c}_6^t  \right] = I_{(n-1)}.$$

\noindent {\textbf{Step 6 : }} $\underline{Y_{ij},  \mbox{ for } i\neq j, \  i,j =2,3,...,b+1}$\\

\begin{small}
Now, $D_5C_4 = \left[
\begin{array}{c|c}
-2(n-5)J_2& J_{2 \times (n-3)}\\
\hline
-2(n-6)J_{(n-3) \times 2} & \textbf{0}_{(n-3)}
\end{array}
\right].$
\end{small} 
Thus, using calculations from Step $5$,  we have
$$Y_{ij} = \frac{1}{2(n-6)} \left[ D_2^tC_2 +D_5C_4 +\mathbf{d}_6\mathbf{c}_6^t \right] = \textbf{0}_{(n-1)}.$$

\noindent {\textbf{Step 7 : }} $\underline{Y_{ij},  \mbox{ for } i =2,3,...,b+1, j= b+2}$\\

\begin{small}
Next, $ D_2^t\mathbf{c}_3 = (n-6)\left[
\begin{array}{c}
-(b-4)\mathds{1}_2 \\
\hline
6\mathds{1}_{(n-3)} 
\end{array}
\right], 
D_4\mathbf{c}_6 = (n-6)\left[
\begin{array}{c}
\mathds{1}_2 \\
\hline
2\mathds{1}_{(n-3)} 
\end{array}
\right] \text{ and } \\
D_5\mathbf{c}_6 = 2(n-6)\left[
\begin{array}{c}
2\mathds{1}_2 \\
\hline
3\mathds{1}_{(n-3)} 
\end{array}
\right].$\end{small} 
Thus, for $ i =2,3,...,b+1, j= b+2$,  we have
$$Y_{ij} = \frac{1}{2(n-6)} \left[ D_2^t\mathbf{c}_3 +D_4\mathbf{c}_6 +(b-1)D_5\mathbf{c}_6 +c_7\mathbf{d}_3 \right] = \textbf{0}_{(n-1) \times 1}.$$

\noindent {\textbf{Step 8 : }} $\underline{Y_{ij},  \mbox{ for } i =b+2, j=1}$\\

\begin{small}Here,  
$
\mathbf{d}_3^tC_1 =\left[
\begin{array}{c|c}
-4b\mathds{1}_2^t  & 2b \mathds{1}_3^t
\end{array}
\right] \mbox{ and }
\mathbf{d}_6^tC_2^t =\left[
\begin{array}{c|c}
4\mathds{1}_2^t  & -2 \mathds{1}_3^t
\end{array}
\right]. 
$\end{small} 
Thus, for $ i =b+2, j=1$,  we have
$$Y_{ij} = \frac{1}{2(n-6)} \left[\mathbf{d}_3^tC_1 +b \mathbf{d}_6^tC_2^t \right] = \textbf{0}_{1 \times 5}.$$

\noindent {\textbf{Step 9 : }} $\underline{Y_{ij},  \mbox{ for } i =b+2, j=2,3,...,b+1}$\\

\begin{small}Next,  
$
\mathbf{d}_3^tC_2 =\left[
\begin{array}{c|c}
2(n-4)\mathds{1}_2^t  & -2 \mathds{1}_{(n-3)}^t
\end{array}
\right] \text{ and }
\mathbf{d}_6^tC_1 =\left[
\begin{array}{c|c}
-2(n-4)\mathds{1}_2^t  & 2 \mathds{1}_{(n-3)}^t
\end{array}
\right].
$\end{small}

\noindent Thus, for $i =b+2, j=2,3,...,b+1$,  we have
$$Y_{ij} = \frac{1}{2(n-6)} \left[\mathbf{d}_3^tC_2 +\mathbf{d}_6^tC_1 \right] = \textbf{0}_{1 \times (n-1)}.$$

\noindent {\textbf{Step 10 : }} $\underline{Y_{ij},  \mbox{ for } i = j = b+2}$\\

\begin{small}Now,  
$
\mathbf{d}_3^t\mathbf{c}_3  = 2(n-6)(2b+1)-6b(n-6) \mbox{ and } 
\mathbf{d}_6^t\mathbf{c}_6  = 2(n-6).
$\end{small}
Thus, for $ i = j = b+2$,  we have
$$Y_{ij} = \frac{1}{2(n-6)} \left[ \mathbf{d}_3^t\mathbf{c}_3  +b \mathbf{d}_6^t\mathbf{c}_6  \right] = 1.$$

From the above calculations, it is clear that $Y=I$ and hence, the desired result follows.
\end{proof}

\section{Conclusion}
In this article, we first find the  cofactor  and determinant of the distance matrix for complete $m$-partite graphs  and then find its inverse whenever it exists. Unlike the case of complete bipartite graphs,  the determinant and cofactor can be zero for infinitely many complete $m$-partite graphs $(m\geq 3) $ and provide an equivalent condition for which the determinant and cofactor of the distance matrix is zero. Next, we consider the distance matrix of multi-block graphs with blocks whose distance  matrices having nonzero cofactor.  For this case, if inverse exists, we find the inverse as a rank one perturbation of a multiple of the Laplacian-like matrix similar to trees, block graphs and bi-block graphs. We also provide the inverse of the distance matrix for a class of multi-block graphs with one block whose distance matrix having zero cofactor.  Consequently, as a special case to multi-block graphs, we compute the determinant and inverse of the distance matrix for a  class of completely positive graphs.\\

\noindent{ \textbf{\Large Acknowledgements}}\\

We take this opportunity to thank the anonymous reviewers for their critical reading of the manuscript  and suggestions which have immensely helped us in getting the article to its present form. The authors would like to  thank one of the anonymous reviewers for the constructive suggestions and inputs for the proof of Propositions~\ref{prop:1} and~\ref{prop:2}. We also thank A.K. Lal for his helpful suggestions and comments. Sumit Mohanty would like to thank the Department of Science and Technology, India, for financial support through the projects MATRICS (MTR/2017/000458).  Some of the computations were verified using the computer package “Sage”. We thank the authors of “Sage” for generously releasing their software as an open-source package.

\small{

}

\end{document}